\documentclass[11pt]{article}

\usepackage[utf8]{inputenc}
\usepackage{pgf,tikz}

\usetikzlibrary{arrows}
\usetikzlibrary{decorations.markings}

\usepackage{latexsym}
\usepackage{amssymb,amsmath,color,mathtools}

\newcommand{\vanish}[1]{}

\usepackage{epsfig}

\usepackage{subfig}

\usepackage{tikz}

\font\Cp = msbm10

\newcommand{\un}{\operatorname{un}}
\newcommand{\inv}{\operatorname{inv}}

\newcommand{\wt}{\operatorname{wt}}
\newcommand{\qchoose}[3]{\genfrac{[}{]}{0pt}{}{#1}{#2}_{#3}}

\newcommand{\disjointunion}{{\stackrel{\centerdot}{\bigcup}}}

\newcommand{\qStirling}[2]{S_{q}[#1,#2]}
\newcommand{\qtStirling}[2]{S_{q,t}[#1,#2]}

\newcommand{\wordposet}{\Pi}

\newcommand{\qed}{\mbox{$\Box$}\vspace{\baselineskip}}

\newcommand{\Allow}{{\mathcal A}}

\newenvironment{proof}{\noindent {\bf Proof:}}{{\qed}}

\newtheorem{theorem}{Theorem}[section]
\newtheorem{goal}{Goal}

\newtheorem{proposition}[theorem]{Proposition}
\newtheorem{lemma}[theorem]{Lemma}
\newtheorem{corollary}[theorem]{Corollary}

\newtheorem{remark}[theorem]{Remark}

\newtheorem{definition}[theorem]{Definition}

\makeatletter
\@addtoreset{equation}{section}
\makeatother

\newcommand{\rg}[2]
{{\mathcal R}(#1,#2)}

\newcommand{\allow}[2]{{\mathcal A}(#1,#2)}

\newcommand{\alwordset}{\mathcal{A}} 
\newcommand{\setpart}{p} 
\newcommand{\Inv}{\operatorname{Inv_r}} 

\font\Cp = msbm10

\newcommand{\Qqq}{\hbox{\Cp Q}}

\newcommand{\Zzz}{\hbox{\Cp Z}}
\newcommand{\Fff}{\hbox{\Cp F}}

\makeatletter
\@addtoreset{equation}{section}
\makeatother

\newcommand{\negwt}{\operatorname{wt}'}

\newcommand{\cupdot}{\dot{\cup}}

\newcommand{\Sym}{{\mathfrak S}}

\newcommand{\rook}[2]
{{\mathcal P}(#1,#2)} 

\newcommand{\alrook}[2]
{{\mathcal Q}(#1,#2)} 

\newcommand{\below}{\operatorname{s}} 

\newcommand{\nrow}{\operatorname{r}} 

\newcommand{\rb}{\operatorname{rb}} 

\newcommand{\lex}{<_{\text{lex}}} 
\usepackage{pgf,tikz}
\usetikzlibrary{arrows}

\newcommand{\prook}[2]
{{\Gamma}(#1,#2)} 

\begin{document}

\title{$q$-Stirling numbers:  A new view}

\author{{\sc Yue CAI}
          \hspace*{2 mm} and \hspace*{2 mm}
        {\sc Margaret A.\ READDY}}

\date{}
\maketitle

\begin{abstract}
We show
the classical $q$-Stirling numbers of the second kind
can be expressed compactly as a pair of statistics
on a subset of restricted growth words.
The resulting expressions are polynomials in
$q$ and $1+q$.
We extend this enumerative result
via a decomposition of a new poset
$\Pi(n,k)$
which we call
the Stirling
poset of the second kind.
Its rank generating function is the
$q$-Stirling number $\qStirling{n}{k}$.
The Stirling poset of the second kind
supports an algebraic complex
and a basis for integer homology is determined.
A parallel
enumerative, poset theoretic and
homological study for the $q$-Stirling numbers of the first
kind is done.
Letting $t = 1+q$ we give a bijective 
argument
showing the
$(q,t)$-Stirling numbers of the first
and second kind are orthogonal.
\end{abstract}

\section{Introduction}
\label{section_introduction}

The idea of $q$-analogues can be traced back 
to Euler in the 1700's who was studying 
$q$-series, especially specializations of theta functions.
The Gaussian polynomial or $q$-binomial
is
the familiar $q$-analogue of the binomial coefficient
given by
$\qchoose{n}{k}{q}  = \frac{[n]_q!}{[k]_q! [n-k]_q!}$,
where
$[n]_q = 1 + q + \cdots + q^{n-1}$
and
$[n]_q ! = [1]_q \cdot [2]_q  \cdots [n]_q$.
A combinatorial interpretation due to MacMahon in 
1916~\cite[Page 315]{MacMahon}
is
$$
   \sum_{\pi \in \Sym(0^{n-k},1^k)} q^{\inv(\pi)} = \qchoose{n}{k}{q}.
$$
Here 
$\Sym(0^{n-k},1^k)$ denotes the number of
$0$-$1$ bit strings  consisting of 
$n-k$ zeroes
and
$k$ ones,
and for
$\pi = \pi_1 \cdots \pi_n \in \Sym(0^{n-k},1^k)$ 
the number of inversions 
is
$\inv(\pi) = | \{(i,j):  i < j \mbox{  and  } \pi_i > \pi_j\}|$.
The inversion statistic goes back to work of
Cramer (1750), B\'ezout (1764) and Laplace (1772).
See the discussion in~\cite[Page 92]{Netto}.
Netto enumerated
the elements of the symmetric group by the inversion
statistic in 1901~\cite[Chapter 4, Sections 54 and 57]{Netto},
and in 1916 MacMahon~\cite[Page 318]{MacMahon} gave the $q$-factorial
expansion
$\sum_{\pi \in \Sym_n} q^{\inv(\pi)} 
= [n]_q !$.

Recent work of 
Fu--Reiner--Stanton--Thiem~\cite[Theorem 1]{Fu_Reiner_Stanton_Thiem}
has expressed the classical $q$-binomial in terms
of a pair of statistics over a {\em subset}
of 
$\Sym(0^{n-k},1^k)$
using powers of $q$ and $1+q$:
\begin{equation}
\label{equation_Stanton_friends}
    \qchoose{n}{k}{q} = \sum_{\pi \in \Omega(n,k)'}
                           q^{a(\omega)} \cdot (1+q)^{p(\omega)}.
\end{equation}
They show this
$q$-$(1+q)$-binomial
is related to Ennola duality for finite unitary groups
and that it counts unitary 
subspaces~\cite[Sections 4 and 6.2]{Fu_Reiner_Stanton_Thiem}.
A two-variable version
exhibits a cyclic sieving 
phenomenon involving unitary 
spaces~\cite[Sections 4 and 5]{Fu_Reiner_Stanton_Thiem}.

It is from the $q$-binomial result~(\ref{equation_Stanton_friends})
 that we 
springboard our work.
Our first goal is enumerative, that is,
to discover more compact encodings of classical
$q$-analogues:

\begin{goal}
Given a $q$-analogue
$$
    f(q) =  \sum_{w \in S} q^{\sigma(w)},
$$
for some statistic
$\sigma(\cdot)$,
find a subset $T \subseteq S$
and statistics
$A(\cdot)$ and $B(\cdot)$ so that
the $q$-analogue
may be expressed as
\begin{equation}
\label{equation_goal_one}
    f(q) = \sum_{w \in T} q^{A(w)} \cdot (1+q)^{B(w)}.
\end{equation}
\label{goal_one}
\end{goal}
For the $q$-Stirling numbers of the first and
second kinds, we develop their $q$-$(1+q)$-analogues.
Furthermore, we are able to
understand these
$q$-$(1+q)$-analogues 
via enumerative, poset theoretic and topological viewpoints.
These leads to the following expanded goal:

\begin{goal}
Given a $q$-analogue
which can be written compactly 
as a $q$-$(1+q)$-analogue 
as in~(\ref{equation_goal_one}),
find poset theoretic and homological reasons
to explain this phenomenon.
\label{goal_two}
\end{goal}

This paper proceeds as follows.
In Section~\ref{section_RG} 
we recall the notion
of restricted growth words or
$RG$-words to encode set partitions. A weighted version
yields the usual $q$-Stirling numbers of the second kind;
see Lemma~\ref{lemma_q_Stirling_second}.
In Section~\ref{section_allowable}
we describe a subset of $RG$-words,
which we call {\em allowable},
whose weighting gives the
$q$-Stirling numbers of the second kind and hence 
a more compact presentation of the $q$-Stirling numbers
of the second kind;  see
Theorem~\ref{theorem_q_Stirling_second_allowable}.

We then take a poset theoretic viewpoint
in Section~\ref{section_Morse}
where we introduce the Stirling poset of the second kind
$\Pi(n,k)$.
Its rank generating function is precisely the
$q$-Stirling number $\qStirling{n}{k}$.
Using discrete Morse theory, we show
in Theorem~\ref{theorem_54} that 
the Stirling poset of the second kind has an acyclic matching.
In Section~\ref{section_topology_of_poset}
we give a decomposition of the Stirling poset into Boolean algebras
with the minimal element of each Boolean algebra
corresponding to an allowable $RG$-word;
see Theorem~\ref{theorem_Boolean_algebra_decomposition}.
A generating function for the $q$-analogue of critical cells
is provided.

In Section~\ref{section_homological_Stembridge}
we review the notion of an algebraic complex supported
on a poset.
In Theorem~\ref{theorem_homology_basis_Stirling_second}
we show that
the Stirling poset $\Pi(n,k)$ supports an algebraic complex
and give a basis for the integer homology,
all of which occurs in even
dimensions.
We give two proofs of this result.  The first
uses  Hersh, Shareshian and
Stanton's homological interpretation of Stembridge's 
$q=-1$ phenomenon,
while the second is an elementary proof using
the poset decomposition in Section~\ref{section_topology_of_poset}.

In Section~\ref{section_first_kind}  we review the
de~M\'edicis--Leroux rook placement interpretation of the $q$-Stirling
numbers of the first kind.  
In Theorem~\ref{theorem_q_Stirling_first} we show
a subset of these boards,
with the appropriate weighting, yields 
a compact
representation of the $q$-Stirling number of the first kind.
In Section~\ref{section_Stirling_first_poset}
we introduce the Stirling poset of the 
first kind~$\Gamma(m,n)$
whose rank generating function is precisely the
$q$-Stirling number $c_q[n,k]$.  Again, a
decomposition of this graded poset is given.
We show the
Stirling poset of the first kind supports an
algebraic complex and 
describe
a basis for the integer homology
which occurs in
even dimensions.
See Theorems~\ref{theorem_decomposition_Stirling_poset_first}
and~\ref{theorem_algebraic_complex_first_kind}.
In Section~\ref{section_orthogonality}
we introduce $(q,t)$-analogues of the Stirling numbers of the
first and second kinds
and show orthogonality holds combinatorially.
We end with concluding remarks.

\section{$RG$-words}
\label{section_RG}

Recall a {\em set partition} 
of the $n$ elements
$\{1, 2, \ldots, n\}$ is a decomposition
of this set into mutually disjoint 
nonempty sets called blocks.
Unless otherwise indicated, throughout
all set partitions will
be written in standard form, that is,
a partition into $k$ blocks will
be denoted by
$\pi = B_1/ B_2/ \cdots / B_k$,
where the blocks are ordered so that 
$\min(B_1) < \min(B_2) < \cdots < \min(B_k)$.
We denote the set of all partitions of $\{1, 2, \ldots, n\}$
by $\Pi_n$.

Given a partition $\pi \in \Pi_n$,
we encode it using a {\em restricted growth word} 
$w(\pi) = w_1 w_2 \cdots w_n$,
where~$w_i = j$ if the element $i $ occurs in the $j$th block $B_j$ of $\pi$. 
For example, the partition
$\pi = 14/236/57$ has $RG$-word
$w = w(\pi) = 1221323$.
Restricted growth words are also known as restricted growth functions.
Recall
a {\em restricted growth function }
$f: \{1, 2, \ldots, n\} \longrightarrow \{1, 2, \ldots, k\}$
is a surjective map which satisfies
$f(1) = 1$ and 
$f(i) \leq \max(f(1), f(2), \ldots, f(i-1)) + 1$
for $i = 2, 3, \ldots, n$.
They
have been studied by 
Hutchinson~\cite{Hutchinson}
and
Milne~\cite{Milne_restricted,Milne_q_analog}.

Two facts about $RG$-words
follow immediately from using the standard form for set partitions.

\begin{proposition}
\label{proposition_RG_word_properties}
The following properties are satisfied by $RG$-words:
\begin{enumerate}

\vspace{-2mm}
\item
\label{$i$}
Any $RG$-word begins with the element $1$.

\item
\label{$ii$}
For an $RG$-word $w$ let 
$\epsilon(j)$ be the smallest  index such
that $w_{\epsilon(j)} = j$.
Then the
$\epsilon(j)$ form an increasing sequence,
that is,
$$
     \epsilon(1) < \epsilon(2) < \cdots.
$$
\end{enumerate}
\end{proposition}

The {\em $q$-Stirling numbers of the second kind} 
are defined by
\begin{equation}
\label{equation_recurrence_for_q_Stirling_second_kind}
   \qStirling{n}{k} = \qStirling{n-1}{k-1} 
                      + [k]_q \cdot \qStirling{n-1}{k},
                            \mbox{     for } 1 \leq k \leq n,
\end{equation}
with boundary conditions
$\qStirling{n}{0} = \delta_{n,0}$ 
and
$\qStirling{0}{k} = \delta_{0,k}$,
where
$\delta_{i,j}$ is the usual Kronecker delta function.
Setting $q=1$ gives the familiar Stirling number of the second kind
$S(n,k)$ which enumerates the number of partitions
$\pi \in \Pi_n$ with exactly $k$ blocks.
There is a long history of studying 
set partition statistics~\cite{Garsia_Remmel,Leroux,Rota}
and $q$-Stirling numbers~\cite{Carlitz,Ehrenborg_Readdy_juggling,Gould,Milne_q_analog,Wachs_White}.

We begin by presenting a statistic on $RG$-words which generates
the $q$-Stirling numbers of the second kind.
Let 
$\rg{n}{k}$ denote the set of all $RG$-words of length $n$ with 
maximum letter $k$, which corresponds to set partitions of $\{1, 2, \ldots, n\}$ into $k$ blocks.
For $w \in \rg{n}{k}$, let $m_{i} = \max(w_1, w_2, \ldots, w_{i})$ and form the weight
$\wt(w) = \prod_{i=1}^n \wt_i(w)$,
where $\wt_1(w) = 1$ and for
$2 \leq i \leq n$, let
\begin{align}
\label{equation_Stirling_second_weight}
\wt_i(w) =
      \left\{  \begin{array}{lc}
              q^{w_i-1}        &\mbox{if $m_{i-1} \geq w_i$},\\
               1             &\mbox{if $m_{i-1} < w_i$}.
               \end{array}
      \right.
\end{align}
For example, 
$\wt(1221323) 
= 1 \cdot 1 \cdot q^1 \cdot q^0 \cdot 1 \cdot q^1 \cdot q^2 = q^4$. 
In terms of set partitions, 
the weight of $\pi = B_1 / B_2 / \cdots / B_k$ is 
$\wt(\pi) = \prod_{i=1}^k q^{(j-1)\cdot (|B_j|-1)}$.

\begin{proposition}
For $w = w_1 \cdots w_n \in \rg{n}{k}$ the weight is given by
\begin{equation*}
     \wt(w) = q^{\sum_{i=1}^{n} w_i  - n - \binom{k}{2}} .
\end{equation*}
\end{proposition}

\begin{lemma}
The $q$-Stirling number of the second kind is given by
$$
   S_q[n,k] = \sum_{w \in \rg{n}{k}} \wt(w).
$$
\label{lemma_q_Stirling_second}
\end{lemma}
\begin{proof}
We show 
$RG$-words $w \in \rg{n}{k}$ satisfy the
recurrence~(\ref{equation_recurrence_for_q_Stirling_second_kind}).
Given an $RG$-word 
$w = w_1 w_2 \cdots w_n \in \rg{n}{k}$, 
consider the map $\varphi$
defined by removing the last letter of the word, 
that is,
$\varphi(w) =  w_1 w_2 \cdots w_{n-1}$.
Clearly $\varphi : \rg{n}{k} \longrightarrow
\rg{n-1}{k-1} \cupdot \rg{n-1}{k}$.
If the only occurrence 
of the maximum  letter~$k$ in the word $w$
is the $n$th position, that is, $w_n = k$, then 
these words
are in bijection with
the set
$\rg{n-1}{k-1}$.
Otherwise, $\varphi(w)$ is of length $n-1$ and all
the letters from $\{1, 2, \ldots, k\}$ occur at least once
in $\varphi(w)$.
In the first case
$\wt(\varphi(w)) = \wt(w)$.
In the second case the letter $k$ occurs more than once in $w$.
Given
$w' = w_1 w_2 \cdots w_{n-1} \in \rg{n-1}{k}$ there
are $k$ possibilities for the $n$th letter $x$ in the inverse image
$\varphi^{-1}(w') = w_1 w_2 \cdots w_{n-1} x$,  namely,
$x \in \{1, 2, \ldots, k\}$.  
Each possibility respectively contributes 
$1, q^1, \ldots,  q^{k-1}$ to the weight,
giving a total weighting contribution of
$[k]_q$.
\end{proof}

See Table~\ref{table_RG} for 
the $RG$-word computation of the $q$-Stirling number $\qStirling{4}{2}$.

\begin{table}[t]
\begin{center}
\begin{tabular}{|c|c|c|}
\hline
Partition &$RG$-word $w$  &$\wt(w)$\\ \hline \hline
$1/234$   &$1222$   &$1\cdot 1 \cdot q \cdot q = q^2$\\ \hline
$12/34$   &$1122$   &$1\cdot 1 \cdot 1 \cdot q = q$ \\ \hline
$13/24$   &$1212$   &$1\cdot 1 \cdot 1 \cdot q = q$\\ \hline
$14/23$   &$1221$   &$1\cdot 1 \cdot q \cdot 1 = q$\\ \hline
$134/2$   &$1211$   &$1\cdot 1 \cdot 1 \cdot 1 = 1$\\ \hline
$124/3$   &$1121$   &$1 \cdot 1 \cdot 1 \cdot 1 = 1$\\ \hline
$123/4$   &$1112$   &$1 \cdot 1 \cdot 1 \cdot 1 = 1$ \\ \hline
\end{tabular}
\end{center}
\caption{Using $RG$-words to compute $S_q[4,2] = q^2 + 3q + 3$.}
\label{table_RG}
\end{table}

\section{Allowable $RG$-words}
\label{section_allowable}

Mirroring $q$-$(1+q)$-binomial,
in this section we
define a subset of $RG$-words
and two statistics
$A(\cdot)$ and $B(\cdot)$ 
which generate
the classical $q$-Stirling number of the second kind 
as a polynomial in $q$ and $1+q$.
We will
see in Sections~\ref{section_Morse} through
\ref{section_homological_Stembridge}
that this has
poset and topological implications.

\begin{definition}
An $RG$-word $w \in \rg{n}{k}$ is
{\em allowable} if every even entry
appears exactly once.
Denote by
$\allow{n}{k}$ the set
of all allowable $RG$-words in
$\rg{n}{k}$.
\end{definition}

Another way to state that
$w \in  \rg{n}{k}$ is
an allowable $RG$-word
is that it is an initial segment of
an infinite word of the form
$$
    w = u_1  \cdot 2 \cdot u_3 \cdot 4 \cdot u_5 \cdots, 
$$
where $u_{2i - 1}$ is a  word on the alphabet
of the odd integers $\{1, 3, \ldots, 2i-1\}$.
In terms of set partitions, an
$RG$-word is allowable if in the corresponding set partition every even indexed
block is a singleton block.
See Table~\ref{table_allowable}.

For an $RG$-word $w = w_1 \cdots w_n $ define
$\negwt(w) = \prod_{i=1}^n \negwt_i(w)$,
where for $m_{i} = \max(w_1, \ldots, w_{i})$
\begin{align}
   \negwt_i(w) =
      \left\{  \begin{array}{ll}
              q^{w_{i}-1} \cdot (1+q)  
              &\mbox{if $m_{i-1} > w_i$}, \\
              q^{w_{i}-1}  
              &\mbox{if $m_{i-1} = w_i$},\\
              1             
              &\mbox{if $m_{i-1} < w_i$ or $i = 1$}.
               \end{array}
      \right.
\end{align}
For completeness, we decompose the $\negwt$
statistic into two  statistics on
$RG$-words.
Let
\begin{align}
   A_i(w) =
      \left\{  \begin{array}{ll}
              w_{i}-1
              &\mbox{if $m_{i-1} \geq w_i$}, \\
              0             
              &\mbox{if $m_{i-1} < w_i$ or $i = 1$},
               \end{array}
      \right.
      \:\:\:\:
      \mbox{     and     }
      \:\:\:\:
   B_i(w) =
      \left\{  \begin{array}{ll}
              1
              &\mbox{if $m_{i-1} > w_i$}, \\
              0             
              &\mbox{otherwise}.
               \end{array}
      \right.
\end{align}
Define 
$$
   A(w) = \sum_{i=1}^n A_i(w)
      \:\:\:\:
  \mbox{      and      }
      \:\:\:\:
   B(w) = \sum_{i=1}^n B_i(w).
$$

\begin{theorem}
\label{theorem_q_Stirling_second_allowable}
The $q$-Stirling numbers of the second kind can be expressed as a weighting
over the set of allowable $RG$-words as follows:
\begin{align}
\label{equation_allowable_Stirling_second_kind}
    \qStirling{n}{k} = \sum_{w \in \allow{n}{k}} \negwt(w)
                     =  \sum_{w \in \allow{n}{k}} q^{A(w)} \cdot (1+q)^{B(w)}.
\end{align}
Hence evaluating the $q$-Stirling number at $q=-1$
gives the number of weakly increasing allowable words in
$\allow{n}{k}$.
\end{theorem}
\begin{proof}
We proceed by induction on $n$ and $k$.
Clearly the result holds 
for $\qStirling{n}{1}$ and 
$\qStirling{n}{n}$ as
the corresponding allowable words 
are $1 1 \cdots 1$ and $12 \cdots n$, each of weight $1$.

For the general case it is enough to show 
that~(\ref{equation_allowable_Stirling_second_kind})
satisfies the defining
relation~(\ref{equation_recurrence_for_q_Stirling_second_kind})
for the $q$-Stirling numbers of the second kind.
We first consider the case when $k$ is even.  
We split the allowable words according to the value of the
last letter, that is,
we write $w = u \cdot w_n$.
Observe that
$\negwt(w) = \negwt(u) \cdot \negwt_n(w)$.
We have
\begin{align*}
     \sum_{w \in \allow{n}{k}} \negwt(w)
     &=
     \sum_{\substack{
                     u \in \allow{n-1}{k-1} \\
                     w_n = k \\
                     m_{n-1} = k-1}}
           \negwt(u) \cdot \negwt_n(w)
     +
     \sum_{\substack{
                     u \in \allow{n-1}{k} \\
                     w_n < k \\
                     m_{n-1} = k}}
           \negwt(u) \cdot \negwt_n(w) 
     \\
     &=  1 \cdot \qStirling{n-1}{k-1} + 
          ((1+q) + q^2 \cdot (1+q) + \cdots + 
          q^{k-2} \cdot (1+q)) \cdot \qStirling{n-1}{k}\\
      &=  \qStirling{n-1}{k-1} + [k]_q \cdot \qStirling{n-1}{k}.
\end{align*} 
where in the second sum the last letter
$w_n$ is odd.
For the case when $k$ is odd there is a similar computation, 
except then there are three cases:
\begin{align*}
     \sum_{w \in \allow{n}{k}}
          \negwt(w)
     &=
     \sum_{\substack{
                     u \in \allow{n-1}{k-1} \\
                     w_n = k  \\
                     m_{n-1}  = k-1}}
           \negwt(u) \cdot \negwt_n(w)
     +
     \sum_{\substack{
                     u \in \allow{n-1}{k-1} \\
                     w_n = k \\
                     m_{n-1} = k}}
           \negwt(u) \cdot \negwt_n(w)
\\
   &  +
     \sum_{\substack{
                     u \in \allow{n-1}{k-1} \\
                     w_n < k \\
                     m_{n-1}  = k}}
           \negwt(u) \cdot \negwt_n (w).
\end{align*}
Here in the second and third sums the last letter $w_n$ is odd.
In both parity cases for $k$, 
the result is equal to the $q$-Stirling 
number of the second kind $\qStirling{n}{k}$, as desired.   
\end{proof}

\begin{table}[tb]
\begin{center}
{\small
\begin{tabular}{|l|lc||l|lc|}
\hline
                &$w$   &$\wt'(w)$
&               &$w$   &$\wt'(w)$\\  \hline \hline
$\allow{1}{1}$  & $1$  & $1$
& $\allow{5}{3}$  & $12311$  &$(1+q)^2$\\   \cline{1-3}
$\allow{2}{1}$  & $11$ &$1$
&                & $12131$    &$(1+q)^2$\\   \cline{1-3} 
$\allow{2}{2}$  & $12$ &$1$
&                & $12113$    &$(1+q)^2$\\ \cline{1-3} 
$\allow{3}{1}$  & $111$ &$1$
&               & $12133$     &$(1+q)\cdot q^2$\\ \cline{1-3} 
$\allow{3}{2}$  & $121$  &$1+q$ 
&                & $12313$    &$(1+q)\cdot q^2$\\ 
                & $112$  &$1$   
&                & $12331$    &$q^2\cdot (1+q)$\\    \cline{1-3} 
$\allow{3}{3}$  & $123$ &$1$
&                & $12333$    &$q^2 \cdot q^2$ \\   \cline{1-3} 
$\allow{4}{1}$  & $1111$ &$1$
&                & $11213$    &$(1+q)$\\   \cline{1-3} 
$\allow{4}{2}$  & $1211$ &$(1+q)^2$
&                & $11231$  &$(1+q)$\\
                & $1121$ &$(1+q)$
&                & $11233$  &$q^2$\\
                & $1112$ & $1$
&                & $11123$   &$1$\\  \hline
$\allow{4}{3}$  & $1213$ & $(1+q)$
&$\allow{5}{4}$  & $12341$  &$(1+q)$\\
                & $1231$ & $(1+q)$
&                & $12343$  &$q^2(1+q)$\\
                & $1233$ & $q^2$
&               & $12134$  &$(1+q)$\\
                & $1123$ & $1$   
&                & $12314$   &$(1+q)$\\   \cline{1-3}
$\allow{4}{4}$  & $1234$  & $1$
&                & $12334$   &$q^2$\\ \cline{1-3}
$\allow{5}{1}$  & $11111$ & $1$
&                & $11234$   &$1$\\ \cline{1-3} \cline{4-6}
$\allow{5}{2}$  & $12111$  & $(1+q)^3$
&$\allow{5}{5}$  & $12345$  &$1$\\   \cline{4-6}
               &$11211$  & $(1+q)^2$
\\
               & $11121$  & $(1+q)$
\\
                & $11112$  & $1$  \\ \cline{1-3}
\end{tabular}
}
\end{center}
\caption{Allowable $RG$-words in $\allow{n}{k}$ 
and their weight for $1 \leq k \leq n \leq 5$.}
\label{table_allowable}
\end{table}

See Table~\ref{table_allowable} for the allowable $RG$-words
for $1 \leq n \leq 5$.

Denote by $a(n,k) = |\alwordset(n,k)|$ the cardinality of allowable words, 
and call it the 
{\em allowable Stirling number of the second kind}.  
The following holds.

\begin{proposition}
The allowable Stirling numbers of the second kind satisfy the recurrence
$$
     a(n,k) = a(n-1, k-1)
                   + \left\lceil k \slash 2 \right\rceil \cdot a(n-1, k)
    \:\:\:\:
    \mbox{   for $n \geq 1$ and $1 \leq k \leq n$},
$$
with the boundary conditions $a(n,0) = \delta_{n,0}$.
\end{proposition}
\begin{proof}
By definition
each allowable word $w \in \alwordset(n, k)$
corresponds to a set partition of $\{1,2, \ldots, n\}$ into
$k$ nonempty subsets where each block with an even label 
has exactly one element in it. Let $\setpart(w)$ be 
the corresponding set partition.

There are two cases.
If $n$ occurs as a singleton block
in $\setpart(w)$, then
after deleting the element $n$ we obtain a 
set partition of the elements $\{1, 2, \ldots, n-1\}$ into $k-1$ blocks.
This corresponds to a word in $\alwordset(n-1, k-1)$.
Otherwise assume the element $n$ occurs in a block  with
more than one element.
We can first build an allowable set partition 
of $\{1, 2, \ldots, n-1\}$ into $k$ blocks and then put 
the element $n$ into one 
of the $k$ blocks. Notice that $n$ can only be placed
into an
odd numbered block,  
so we have $\lceil k \slash 2 \rceil$ 
possible blocks to assign the element $n$. This gives 
$\lceil k\slash 2 \rceil \cdot a(n-1, k)$ possibilities.
\end{proof}

\begin{table}[tb]
\begin{center}
{\small
\begin{tabular}{|r|rrrrrrrrrrr|r|r|}
\hline
$n\backslash k$   &0  &1  &2  &3  &4  &5  &6  &7  &8  &9  &10 
&$a(n)$  &$b(n)$\\
\hline
0   &1  &&&&&&&&&&  &1 &1\\
1   &0  &1  &&&&&&&&&   &1 &1\\
2   &0  &1  &1  &&&&&&&&    &2 &2\\
3   &0  &1  &2  &1  &&&&&&&     &4 &5\\
4   &0  &1  &3  &4  &1  &&&&&&      &9 &15\\
5   &0  &1  &4  &11 &6  &1  &&&&&       &23 &52\\
6   &0  &1  &5  &26 &23 &9  &1  &&&&        &65 &203\\
7   &0  &1  &6  &57 &72 &50 &12 &1  &&&         &199 &877\\
8   &0  &1  &7  &120    &201    &222    &86 &16 &1  &&  &654 &4140\\
9   &0  &1  &8  &247    &522    &867    &480    &150    &20 &1  &   &2296 &21147\\
10  &0  &1  &9  &502    &1291    &3123   &2307   &1080   &230    &25 &1  &8569 &115975\\
\hline
\end{tabular}
}
\caption{The allowable Stirling numbers of the second kind
$a(n,k)$, the
allowable Bell numbers $a(n)$ 
and the classical Bell numbers $b(n)$ for $0 \leq n \leq 10$.}
\label{table_allowable_Stirling_Bell}
\end{center}
\end{table}

We call the sum 
$a(n) = \sum_{k=0}^n a(n,k)$ the
{\em $n$th allowable Bell number}.
See Table~\ref{table_allowable_Stirling_Bell}.
The following properties are straightforward to verify.

\begin{proposition}
The allowable Stirling numbers of the second kind satisfy
\begin{align}
\label{equation_stir2_2}
	a(n,2) &= n-1\,, \\
\label{equation_stir2_n}
    a(n,n-1) &= \Big\lfloor \frac{n}{2} \Big\rfloor \cdot \Big\lceil \frac{n}{2} \Big\rceil\,.
\end{align}
\end{proposition}
\begin{proof}
By definition
any $w \in \alwordset(n, 2)$
is a word of length $n$ consisting of exactly $n-1$ 1's and one~$2$.
Since the initial letter must be $1$, there are $n-1$ choices to
assign the location of $2$.
Thus~(\ref{equation_stir2_2}) follows.

For identity~(\ref{equation_stir2_n}) we wish 
to count allowable words of length $n$ with 
maximal entry $n-1$. 
By definition of an allowable word, 
there will be exactly one odd integer that appears twice 
and all other integers appear exactly once in such a word. 
In other words, 
given the word $12\cdots (n-1)$,
we need to insert an odd integer less than or equal to
$n-1$ so that the resulting word is still allowable.
There are  $\lceil (n-1) \slash 2 \rceil = \lceil n \slash 2 \rceil$ 
choices for such an odd integer. 
We can place this odd integer anywhere 
after its initial appearance in the word $12\cdots (n-1)$.
Thus
we have in total 
$(n-1) + (n-3) + \cdots 
+ (n-(2 \cdot \lceil (n-1) \slash 2 \rceil -1)) = 
\lfloor n \slash 2 \rfloor \cdot \lceil n \slash 2 \rceil$
ways to obtain a word in $\alwordset(n, n-1)$.
\end{proof}

Homological underpinnings
of Theorem~\ref{theorem_q_Stirling_second_allowable}
will be discussed in 
Section~\ref{section_homological_Stembridge}.

\section{The Stirling poset of the second kind}
\label{section_Morse}

In order to understand the $q$-Stirling numbers
more deeply, we give a 
poset structure on $\mathcal{R}(n,k)$,
which we
call {\em the Stirling poset of the second kind},
denoted by $\Pi(n,k)$, as follows. 
For $v, w \in \mathcal{R}(n,k)$  let
$ v = v_1 v_2 \cdots v_n \prec w$
if 
$w = v_1 v_2 \cdots(v_i+1) \cdots v_n$
for some index $i$. It is clear that if
$v \prec w$
then $\wt(w)=q\cdot \wt(v)$,
where the weight is as defined in~(\ref{equation_Stirling_second_weight}).
The Stirling poset of the second kind is graded
by the degree of the weight function $\wt$.
Thus
the rank of the poset
$\Pi(n,k)$ is $(n-k)(k-1)$ and its rank generating function is given
by $S_q[n,k]$.
For basic terminology regarding posets, we refer the reader to
Stanley's treatise~\cite[Chapter 3]{Stanley_EC_I}.
See Figures~\ref{figure_Pi_5_2}
and~\ref{figure_Pi_5_3} for two examples of the Stirling poset
of the second kind.

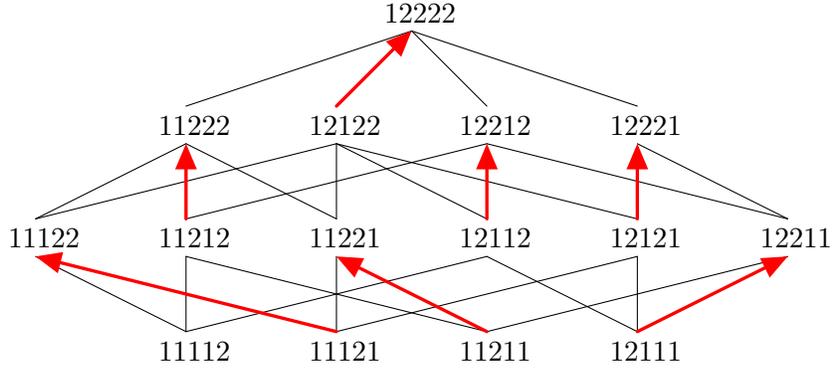
\begin{figure}
[t]
\begin{center}
\begin{tikzpicture}[line cap=round,line join=round,>=triangle 45,x=1.0cm,y=1.0cm]
\clip(-4.5,1.0) rectangle (7.5,6.5);
\draw (1,6) node[anchor=north west] {$12222$};
\draw (-2,4.5) node[anchor=north west] {$11222$};
\draw (0,4.5) node[anchor=north west] {$12122$};
\draw (2,4.5) node[anchor=north west] {$12212$};
\draw (4,4.5) node[anchor=north west] {$12221$};
\draw (-4,3) node[anchor=north west] {$11122$};
\draw (-2,3) node[anchor=north west] {$11212$};
\draw (0,3) node[anchor=north west] {$11221$};
\draw (2,3) node[anchor=north west] {$12112$};
\draw (4,3) node[anchor=north west] {$12121$};
\draw (6,3) node[anchor=north west] {$12211$};
\draw (-2,1.5) node[anchor=north west] {$11112$};
\draw (0,1.5) node[anchor=north west] {$11121$};
\draw (2,1.5) node[anchor=north west] {$11211$};
\draw (4,1.5) node[anchor=north west] {$12111$};
\draw (1.5,5.5)-- (-1.5,4.5);
\draw (1.5,5.5)-- (2.5,4.5);
\draw (0.5,4.0)-- (2.5,3.0);
\draw (1.5,5.5)-- (4.5,4.5);
\draw (-1.5,4.0)-- (-3.5,3.0);
\draw (-1.5,4.0)-- (0.5,3.0);
\draw (0.5,4.0)-- (-3.5,3.0);
\draw (0.5,4.0)-- (4.5,3.0);
\draw (2.5,4.0)-- (6.5,3.0);
\draw (2.5,4.0)-- (-1.5,3.0);
\draw (0.5,4.0)-- (0.5,3.0);
\draw (4.5,4.0)-- (6.5,3.0);
\draw (-3.5,2.5)-- (-1.5,1.5);
\draw (-1.5,2.5)-- (-1.5,1.5);
\draw (-1.5,2.5)-- (2.5,1.5);
\draw (0.5,2.5)-- (0.5,1.5);
\draw (2.5,2.5)-- (-1.5,1.5);
\draw (2.5,2.5)-- (4.5,1.5);
\draw (4.5,2.5)-- (0.5,1.5);
\draw (4.5,2.5)-- (4.5,1.5);
\draw (6.5,2.5)-- (2.5,1.5);
\draw [->,line width=1.25pt,color=red] (0.5,4.5) -- (1.5,5.5);
\draw [->,line width=1.25pt,color=red] (-1.5,3.0) -- (-1.5,4.0);
\draw [->,line width=1.25pt,color=red] (2.5,3.0) -- (2.5,4.0);
\draw [->,line width=1.25pt,color=red] (4.5,3.0) -- (4.5,4.0);
\draw [->,line width=1.25pt,color=red] (0.5,1.5) -- (-3.5,2.5);
\draw [->,line width=1.25pt,color=red] (2.5,1.5) -- (0.5,2.5);
\draw [->,line width=1.25pt,color=red] (4.5,1.5) -- (6.5,2.5);
\end{tikzpicture}
\end{center}
\caption{
The matching of the Stirling poset $\Pi(5,2)$.} 
\label{figure_Pi_5_2}
\end{figure}

We next review the notion of a 
Morse matching~\cite{Kozlov_discrete_Morse, Dmitry_Kozlov}.
This will enable us to find a natural decomposition of
the Stirling poset of the second kind,
and to later be able to draw homological conclusions.
A {\em partial matching} on a poset
$P$ is a matching on the underlying
graph of the Hasse diagram of $P$, that is, 
a subset $M \subseteq P \times P$ satisfying
($i$) the ordered pair
$(a, b) \in M$ implies $a \prec b$,
and
($ii$)
each element $a \in P$ belongs to at most one element
in $M$.
When $(a, b) \in M$, we write
$u(a) = b$ and $d(b) = a$.
A partial matching on $P$ is  {\em acyclic}
if there does not exist a cycle
$$
    a_1 \prec u(a_1) \succ a_2 \prec u(a_2) 
        \succ \cdots \succ a_n \prec u(a_n) \succ a_1
$$
with $n \geq 2$, and the elements $a_1, a_2, \ldots, a_n$ 
are distinct.

An alternate manner is to orient all the edges in the Hasse
diagram of a poset downwards and then reorient all the edges
occurring in the matching upwards.  The acyclic condition
is simply that there is no cycle on the directed
Hasse diagram.  
For the matched edge $(a,b)$ the notation
$u(a) = b$ and $d(b) = a$ denotes the fact that in the edge oriented
from $a$ to $b$ the element $b$ is ``upwards'' from~$a$
and similarly the element $a$ is ``downwards'' from $b$.
One can use the terminology of a 
{\em gradient path} or
{\em $V$-path} consisting
alternatively of matched and unmatched elements from 
the poset~\cite{Forman_Morse_theory}.
A {\em discrete Morse matching} is one where no gradient path
forms a cycle.

We define  a matching $M$ on
the Stirling poset  $\wordposet(n,k)$ 
in the following manner.
Let $w_i$ be the first entry in
$w = w_1 w_2 \cdots w_n \in \mathcal{R}(n,k)$
such that $w$ is weakly decreasing, that is, 
$w_1 \leq w_2 \leq \cdots \leq w_{i-1} \geq w_i$
and where we require the inequality $w_{i-1} \geq w_i$
to be strict unless both $w_{i-1}$ and $w_i$
are even. We have two subcases.
If $w_i$ is even 
then let
$d(w) = w_1 w_2 \cdots w_{i-1} (w_i-1) w_{i+1} \cdots w_n$.
In this case
we have
$\wt(d(w)) = q^{-1} \cdot \wt(w)$.
Otherwise, if $w_i$ is odd then let
$u(w) = w_1 w_2 \cdots w_{i-1} (w_i+1) w_{i+1} \cdots w_n$
and we have
$\wt(u(w)) = q \cdot \wt(w)$.
If $w$ is an allowable word which is weakly increasing, then
$w$ is unmatched in the poset.
Again, we refer to 
Figures~\ref{figure_Pi_5_2}
and~\ref{figure_Pi_5_3}.

\begin{figure}[t]
\begin{center}
\begin{tikzpicture}[line cap=round,line join=round,>=triangle 45,x=1.0cm,y=1.0cm, scale = 0.60]
\clip(-5.5,-4.0) rectangle (11.5,7.5);
\draw (2.0,7.0) node[anchor=north west] {12333};
\draw (-2.0,4.5) node[anchor=north west] {12233};
\draw (2.0,4.5) node[anchor=north west] {12323};
\draw (6.0,4.5) node[anchor=north west] {12332};
\draw (-4.0,2.0) node[anchor=north west] {11233};
\draw (0.0,2.0) node[anchor=north west] {12223};
\draw (2.0,2.0) node[anchor=north west] {12232};
\draw (4.0,2.0) node[anchor=north west] {12313};
\draw (6.0,2.0) node[anchor=north west] {12322};
\draw (8.0,2.0) node[anchor=north west] {12331};
\draw (-2.0,2.0) node[anchor=north west] {12133};
\draw (-5.0,-0.5) node[anchor=north west] {11223};
\draw (-3.0,-0.5) node[anchor=north west] {11232};
\draw (-1.0,-0.5) node[anchor=north west] {12123};
\draw (3.0,-0.5) node[anchor=north west] {12213};
\draw (5.0,-0.5) node[anchor=north west] {12231};
\draw (7.0,-0.5) node[anchor=north west] {12312};
\draw (9.0,-0.5) node[anchor=north west] {12321};
\draw (1.0,-0.5) node[anchor=north west] {12132};
\draw (-3.0,-3.0) node[anchor=north west] {11123};
\draw (-1.0,-3.0) node[anchor=north west] {11213};
\draw (1.0,-3.0) node[anchor=north west] {11231};
\draw (3.0,-3.0) node[anchor=north west] {12113};
\draw (5.0,-3.0) node[anchor=north west] {12131};
\draw (7.0,-3.0) node[anchor=north west] {12311};

\draw (2.7,6.4)-- (-1.2,4.5);
\draw (2.7,6.4)-- (2.7,4.5);
\draw (2.7,6.4)-- (6.7,4.5);

\draw (-1.3,3.9)--(-3.3,2.2);
\draw (-1.3,3.9)-- (0.7,2.0);
\draw (-1.3,3.9)-- (2.7,2.0);
\draw (2.7,3.9)-- (0.7,2.0);
\draw (2.7,3.9)-- (6.7,2.0);
\draw (6.7,3.9)-- (2.7,2.0);
\draw (6.7,3.9)-- (6.7,2.0);
\draw [->,line width=1.25pt,color=red] (4.7,2.0) -- (2.7,3.9);
\draw [->,line width=1.25pt,color=red] (8.7,2.0) -- (6.7,3.9);
\draw [->,line width=1.25pt,color=red] (-1.3,2.0) -- (-1.3,3.9);

\draw (-3.3,1.4)-- (-4.3,-0.5);
\draw (-3.3,1.4)-- (-2.3,-0.5);
\draw (-1.3,1.4)-- (-0.3,-0.5);
\draw (-1.3,1.4)-- (1.7,-0.5);
\draw (0.7,1.4)-- (-4.3,-0.5);
\draw (0.7,1.4)-- (3.7,-0.5);
\draw (2.7,1.4)-- (-2.3,-0.5);
\draw (2.7,1.4)-- (5.7,-0.5);
\draw (4.7,1.4)-- (3.7,-0.5);
\draw (4.7,1.4)-- (7.7,-0.5);
\draw (6.7,1.4)-- (9.7,-0.5);
\draw (8.7,1.4)-- (5.7,-0.5);
\draw (8.7,1.4)-- (9.7,-0.5);
\draw [->,line width=1.25pt,color=red] (-0.3,-0.5) -- (0.7,1.4);
\draw [->,line width=1.25pt,color=red] (1.7,-0.5) -- (2.7,1.4);
\draw [->,line width=1.25pt,color=red] (7.7,-0.5) -- (6.7,1.4);

\draw (-4.1,-1.1)-- (-2.3,-3);
\draw (-0.3,-1.1)-- (-2.3,-3);
\draw (-0.3,-1.1)-- (3.7,-3.0);
\draw (1.7,-1.1)-- (5.7,-3.0);
\draw (3.7,-1.1)-- (-0.3,-3.0);
\draw (5.7,-1.1)-- (1.7,-3.0);
\draw (7.7,-1.1)-- (7.7,-3);
\draw [->,line width=1.25pt,color=red] (7.7,-3.0) -- (9.7,-1.1);
\draw [->,line width=1.25pt,color=red] (-0.3,-3.0) -- (-4.1,-1.1);
\draw [->,line width=1.25pt,color=red] (1.7,-3.0) -- (-2.3,-1.1);
\draw [->,line width=1.25pt,color=red] (3.7,-3.0) -- (3.7,-1.1);
\draw [->,line width=1.25pt,color=red] (5.7,-3.0) -- (5.7,-1.1);
\end{tikzpicture}
\end{center}

\caption{
The Stirling poset $\Pi(5,3)$
and its discrete Morse matching.  
The rank generating function is the
$q$-Stirling number
$S_q[5,3] = q^4 + 3 q^3 + 7 q^2 + 8 q + 6$.
The matched elements are
indicated by arrows. The unmatched elements are $11123$, $11233$ and $12333$,
and the sum of their weights is $1 + q^2 + q^4$.}
\label{figure_Pi_5_3} 
\end{figure}
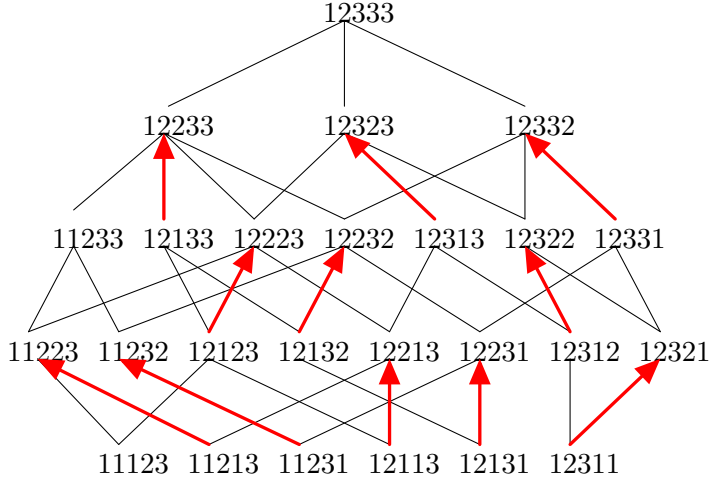

\begin{lemma}
For the  partial matching $M$ described on the poset $\wordposet(n,k)$ the
unmatched words $U(n,k)$ are of the form
\begin{align*}
     w =
         \begin{cases}
             u_1 \cdot 2 \cdot u_3 \cdot 4 \cdot u_5 \cdot 6 \cdots u_{k-1} 
             \cdot k
             &\text{    for $k$ even},\\
             u_1 \cdot 2 \cdot u_3 \cdot 4 \cdot u_5 \cdot 6 \cdots (k-1) 
             \cdot u_{k} 
             &\text{    for $k$ odd},
         \end{cases}
\end{align*}
where 
$u_{2i-1} = (2i-1)^{j_i}$, 
that is,
$u_{2i-1}$
is a word consisting of $j_i \geq 1$ copies of the odd integer $2i-1$.
\end{lemma}
\begin{proof}
The result follows
by observing
the unmatched elements of the Stirling poset $w(n,k)$ consist
of $RG$-words in ${\mathcal R}(n,k)$ which are always increasing
and have no repeated even-valued entries.  
\end{proof}

\begin{lemma}
\label{lemma_gradient_new}
Let $a$ and $b$ be two distinct elements in the Stirling poset of the
second kind $\Pi(n,k)$ such that
$a \prec u(a) \succ b \prec u(b)$.
Then the element $a$ is lexicographically larger than the
element $b$.
\end{lemma}
\begin{proof}
Suppose on the contrary that
$a \lex b$
with 
$a = a_1 \cdots a_n$.
Assume that 
$u(a) = a_1 a_2 \cdots (a_i+1) \cdots a_n$.
Then~$a_i$ is odd and the strict inequality $a_{i-1} > a_i$ holds.
Since $a$ is lexicographically smaller than
$b$
and
the element~$b$ is obtained by decreasing an entry in $u(a)$ by one, 
the element $b$ must be of the form 
$b = a_1 \cdots (a_i + 1) \cdots (a_j-1) \cdots a_n$ 
for some index $j > i$. 
The first $i$ entries in $b$ 
satisfy  
$a_1 \leq a_2 \leq \cdots \leq a_{i-1} \geq (a_i + 1)$
and $a_i + 1$ is even,
so by 
definition the element $b$ is matched to an element of 
lower rank,
contradicting the fact that 
$(b, u(b))$ is a matched pair in $M$. 
\end{proof}

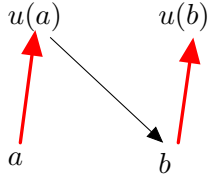
\begin{figure}[t]
\centering
\begin{tikzpicture}[line cap=round,line join=round,>=triangle 45,x=1.0cm,y=1.0cm]
\clip(-3.5,2.5) rectangle (0,5.5);
\draw (-3,5) node[anchor=north west] {$u(a)$};
\draw (-1,5) node[anchor=north west] {$u(b)$};
\draw (-3,3) node[anchor=north west] {$a$};
\draw (-1,3) node[anchor=north west] {$b$};
\draw [->,line width=1.25pt,color=red] (-2.7,3) -- (-2.5,4.5);
\draw [->] (-2.3,4.4) -- (-0.8,3);
\draw [->,line width=1.25pt,color=red] (-0.6,3) -- (-0.4,4.4);
\end{tikzpicture}
\caption{First three steps of a gradient path.} \label{gradient_path_2}
\end{figure}

\begin{theorem}
\label{theorem_54}
The matching $M$ described for  $\wordposet(n,k)$
is an acyclic matching, that is,
it is a discrete Morse matching.
\label{theorem_second_kind_acyclic_matching}
\end{theorem}
\begin{proof}
By Lemma~\ref{lemma_gradient_new}
one cannot find  a gradient cycle
of the form
$$x_1 \prec u(x_1) \succ x_2 \prec u(x_2) \succ \cdots \succ x_k
\prec u(x_k) \succ x_1$$
since the elements 
$x_1, \ldots, x_k$ must
satisfy
$x_1 >_{\rm lex} x_2 >_{\rm lex} \cdots >_{\rm lex}x_k >_{\rm lex} x_1$,
which is impossible.
\end{proof}

We end this section with enumeration of the words which are
left unmatched in the discrete Morse matching.
We will see in Section~\ref{section_homological_Stembridge}
that the unmatched words will provide
a basis for the integer homology of the algebraic complex supported
by the Stirling poset of the second kind.

\begin{lemma}
\label{lemma_generating_function}
The weighted generating function of the unmatched words 
$U(n,k)$ in
$\Pi(n,k)$ is given by the $q^2$-binomial coefficient
$$
    \sum_{u \in U(n,k)} \wt(u) =
     \qchoose{n-1 - \lfloor{\frac{k}{2} }\rfloor}
             {\lfloor \frac{k-1}{2} \rfloor}{q^2}.
$$
\end{lemma}
\begin{proof}
Let $u = u_1 \cdots u_n \in U(n,k)$
be an unmatched word.  
Recall
the weight is given by reading the word from left to right
and gaining a multiplicative  factor  $q^{u_i-1}$ 
for all values of $i$ with $u_{i-1} = u_i$.
Since $u_{i-1} = u_i$ can only appear when 
$u_i$ is odd,
the weight of an unmatched word is always $q^{2m}$ for some
non-negative integer $m$.

We claim that each $u \in U(n,k)$
of weight $q^{2m}$
corresponds to an integer partition of $2m$ with at most $n-k$
parts where each part is even and where each part is at most 
$\rho = \lfloor (k-1) \slash 2 \rfloor\cdot 2$.
The correspondence is as follows. 
For each word $u$ satisfying the condition
with the odd integer $j$ appearing $m_j$ times,
map these odd integers to $m_j - 1$ copies of $j-1$.
The resulting partition of $2m$ is of the form
$$
     2m = \underbrace{2 + \cdots +2}_{m_3 - 1} + 
          \underbrace{4 + \cdots + 4}_{m_5 - 1}
          + \cdots +    
          \underbrace{\rho + \cdots + \rho}_{m_{\sigma} -1},
$$
where $\sigma$ is the largest occurring odd integer
in the original $RG$-word $u$
and $\rho = \sigma - 1$.
For example, the word $112333455$ corresponds to the partition 
$8 = 2+2+4$.
Note that the unmatched word $1$ corresponds to 
the empty partition $\emptyset$.

An alternate way to describe these partitions is to form
a partition  of $m$
into at most $n-k$ parts with each part at most 
$\lfloor (k-1)\slash 2 \rfloor$. 
By doubling each part, we obtain the above mentioned partition. 
However, by~\cite[Proposition 1.7.3]{Stanley_EC_I} 
the sum of the weight of partitions that fit into a rectangle of 
size~$n-k$ by~$\lfloor (k-1)\slash 2 \rfloor$ 
is given by the Gaussian polynomial 
$\qchoose{\lfloor\frac{k-1}{2}\rfloor+n-k}{\lfloor\frac{k-1}{2}\rfloor}{q}$. By the substitution~$q \mapsto q^2$, the result follows.
\end{proof}

\begin{corollary}
The number of unmatched words of length $n$
that is,
$U(n) = \sum_{k=1}^n |U(n,k)|$
is given by the Fibonacci number $F_n$,
where $F_n = F_{n-1} + F_{n-2}$ for
$n \geq 2$ and
$F_0 = F_1 = 1$.
\end{corollary}
\begin{proof}
Substituting $q^2 = 1$, that is,
$q=-1$
in Lemma~\ref{lemma_generating_function}
gives the number of unmatched words 
$|U(n,k)|$ in
the Stirling poset of the second $\Pi(n,k)$.
Hence,
\begin{equation*}
     U(n) = \sum_{k=1}^n |U(n,k)| 
     = \sum_{i=0}^{\lfloor \frac{n}{2} \rfloor} \binom{n-i}{i} = F_{n},
\end{equation*}
where the last equality is a well-known 
binomial coefficient expansion
for the Fibonacci number $F_n$
arising from compositions of $n$ using $1$s and $2$s.
\end{proof}

\section{Decomposition of the Stirling poset of the second kind}
\label{section_topology_of_poset}

We next decompose 
the Stirling poset $\wordposet(n,k)$ into 
Boolean algebras indexed by the allowable words.
This gives a poset 
explanation for the factorization of the
$q$-Stirling number $\qStirling{n}{k}$ in terms
of powers of $q$ and $1+q$.

To state this decomposition,
we need two definitions.
For $w \in \alwordset(n, k)$ an allowable word
let
$\Inv(w) = \{i : w_j > w_i \text{ for some } j < i\}$
be the set  of all indices in $w$ that 
contribute to the right-hand element of an inversion pair. 
For $i \in \Inv(w)$ such an entry $w_i$ must be odd
since in a given allowable word
any entry occurring to the left of an even entry must be strictly
less than it.
Finally, for $w \in \alwordset(n,k)$ let
$\alpha(w)$
be the word formed by incrementing
each of the entries indexed by the
set $\Inv(w)$ by one.
Additionally,
for $w \in \Allow(n,k)$
and any $I \subseteq \Inv(w)$,
the word formed by incrementing each of the entries indexed
by the set $I$ by one
are elements of ${\cal R}(n,k)$
since if $i \in\Inv(w)$ then there is
an index $h < i$ with $w_h = w_i$.
This follows from 
Proposition~\ref{proposition_RG_word_properties} part ($ii$).

\begin{theorem}
\label{theorem_Boolean_algebra_decomposition}
The Stirling poset of the second kind
$\Pi(n,k)$ can be decomposed as the disjoint
union of Boolean intervals 
$$
    \Pi(n,k) = \underset{w \in \Allow(n,k)}\disjointunion 
                    [w, \alpha(w)].
$$
Furthermore, if an allowable word
$w \in \Allow(n,k)$ has weight
$\wt'(w) = q^i \cdot (1+q)^j$,
then 
the rank of the element $w$ is $i$ and the interval
$[w, \alpha(w)]$ is isomorphic to the Boolean algebra on
$j$ elements.
\end{theorem}

\begin{proof}
Let $w \in \Allow(n,k)$
with $\wt'(w) = q^i \cdot (1+q)^j$ and
$|\Inv(w)| = m$.
It directly follows from the definitions
that the interval
$[w, \alpha(w)]$
is isomorphic to the Boolean algebra
$B_m$.
With the exception of the
element $w$, all the other elements in
the interval
$[w, \alpha(w)]$
are not allowable words in $\wordposet(n,k)$
since all of the newly incremented entries will
have at least two equal even entries.
We also claim $m = j$, since $\wt'(w)$ picks
up a factor of $1+q$ for each index 
$i$ satisfying
$w_i < m_{i-1} = \max(w_1, \ldots, w_{i-1})$.
These indices are exactly the set
$\Inv(w)$.

We claim 
every element of
$\wordposet(n,k)$ occurs in some Boolean algebra
in the decomposition.
This is vacuously true if $w \in \Allow(n,k)$.
Otherwise since $w$ is not an allowable word, it has
even entries which are repeated.  Decrease all occurrences
of these repeated entries by one except for the first
occurrence of each even integer.
This is the allowable $RG$-word associated to $w$.
\end{proof}

See 
Figures~\ref{figure_pi_5_2_decomposition} and~\ref{figure_pi_5_3_decomposition}
for examples of this
decomposition for the posets in 
Figures~\ref{figure_Pi_5_2} and~\ref{figure_Pi_5_3}, respectively.


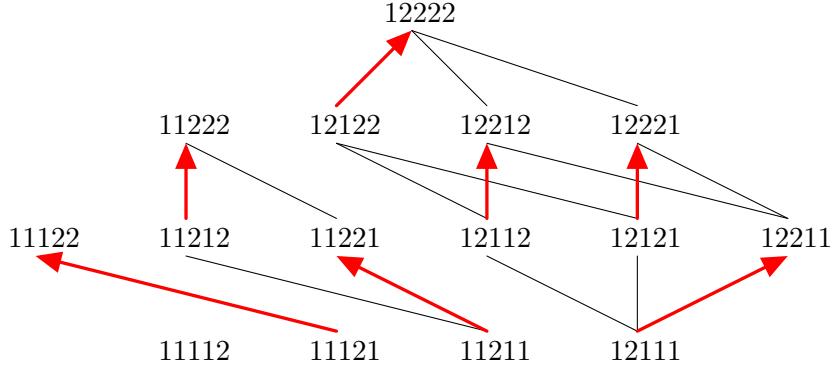
\begin{figure}[t]
\begin{center}
\begin{tikzpicture}[line cap=round,line join=round,>=triangle 45,x=1.0cm,y=1.0cm]
\clip(-4.5,1.0) rectangle (7.5,7);
\draw (1,6) node[anchor=north west] {$12222$};
\draw (-2,4.5) node[anchor=north west] {$11222$};
\draw (0,4.5) node[anchor=north west] {$12122$};
\draw (2,4.5) node[anchor=north west] {$12212$};
\draw (4,4.5) node[anchor=north west] {$12221$};
\draw (-4,3) node[anchor=north west] {$11122$};
\draw (-2,3) node[anchor=north west] {$11212$};
\draw (0,3) node[anchor=north west] {$11221$};
\draw (2,3) node[anchor=north west] {$12112$};
\draw (4,3) node[anchor=north west] {$12121$};
\draw (6,3) node[anchor=north west] {$12211$};
\draw (-2,1.5) node[anchor=north west] {$11112$};
\draw (0,1.5) node[anchor=north west] {$11121$};
\draw (2,1.5) node[anchor=north west] {$11211$};
\draw (4,1.5) node[anchor=north west] {$12111$};
\draw (1.5,5.5)-- (2.5,4.5);
\draw (0.5,4.0)-- (2.5,3.0);
\draw (1.5,5.5)-- (4.5,4.5);
\draw (-1.5,4.0)-- (0.5,3.0);
\draw (0.5,4.0)-- (4.5,3.0);
\draw (2.5,4.0)-- (6.5,3.0);
\draw (4.5,4.0)-- (6.5,3.0);
\draw (-1.5,2.5)-- (2.5,1.5);
\draw (2.5,2.5)-- (4.5,1.5);
\draw (4.5,2.5)-- (4.5,1.5);
\draw [->,line width=1.25pt,color=red] (0.5,4.5) -- (1.5,5.5);
\draw [->,line width=1.25pt,color=red] (-1.5,3.0) -- (-1.5,4.0);
\draw [->,line width=1.25pt,color=red] (2.5,3.0) -- (2.5,4.0);
\draw [->,line width=1.25pt,color=red] (4.5,3.0) -- (4.5,4.0);
\draw [->,line width=1.25pt,color=red] (0.5,1.5) -- (-3.5,2.5);
\draw [->,line width=1.25pt,color=red] (2.5,1.5) -- (0.5,2.5);
\draw [->,line width=1.25pt,color=red] (4.5,1.5) -- (6.5,2.5);
\end{tikzpicture}
\end{center}

\caption{
The decomposition of the Stirling poset $\Pi(5,2)$ into 
Boolean algebras $B_i$ for $i = 0, 1, 2, 3$.
Arrows indicate the elements matched from
the discrete Morse matching.
Based on the ranks of the 
minimal elements in each Boolean algebra,
one obtains the weight of the poset is 
$\qStirling{5}{2}= 1+(1+q)+(1+q)^2+(1+q)^3$.}
\label{figure_pi_5_2_decomposition} 
\end{figure}

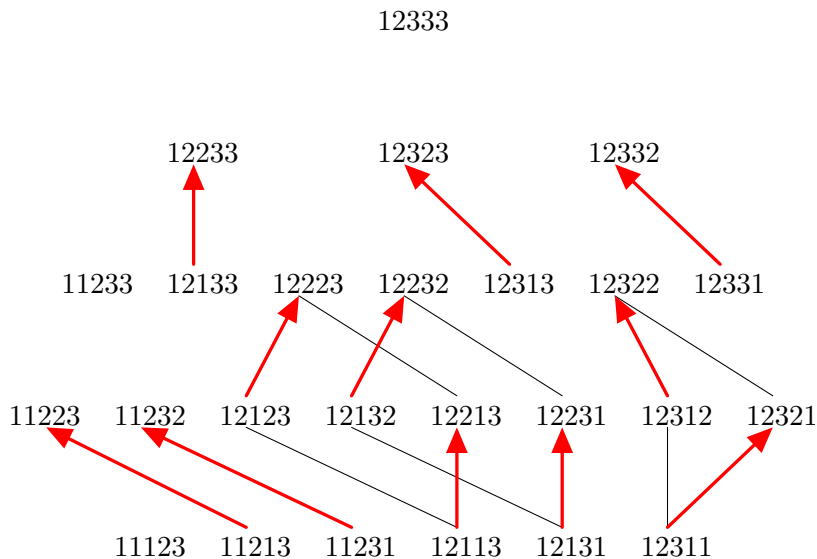
\begin{figure}[htb]
\begin{center}
\begin{tikzpicture}[line cap=round,line join=round,>=triangle 45,x=1.0cm,y=1.0cm, scale = 0.70]
\clip(-5.5,-4.0) rectangle (10.5,7.5);
\draw (2.0,7.0) node[anchor=north west] {12333};
\draw (-2.0,4.5) node[anchor=north west] {12233};
\draw (2.0,4.5) node[anchor=north west] {12323};
\draw (6.0,4.5) node[anchor=north west] {12332};
\draw (-4.0,2.0) node[anchor=north west] {11233};
\draw (0.0,2.0) node[anchor=north west] {12223};
\draw (2.0,2.0) node[anchor=north west] {12232};
\draw (4.0,2.0) node[anchor=north west] {12313};
\draw (6.0,2.0) node[anchor=north west] {12322};
\draw (8.0,2.0) node[anchor=north west] {12331};
\draw (-2.0,2.0) node[anchor=north west] {12133};
\draw (-5.0,-0.5) node[anchor=north west] {11223};
\draw (-3.0,-0.5) node[anchor=north west] {11232};
\draw (-1.0,-0.5) node[anchor=north west] {12123};
\draw (3.0,-0.5) node[anchor=north west] {12213};
\draw (5.0,-0.5) node[anchor=north west] {12231};
\draw (7.0,-0.5) node[anchor=north west] {12312};
\draw (9.0,-0.5) node[anchor=north west] {12321};
\draw (1.0,-0.5) node[anchor=north west] {12132};
\draw (-3.0,-3.0) node[anchor=north west] {11123};
\draw (-1.0,-3.0) node[anchor=north west] {11213};
\draw (1.0,-3.0) node[anchor=north west] {11231};
\draw (3.0,-3.0) node[anchor=north west] {12113};
\draw (5.0,-3.0) node[anchor=north west] {12131};
\draw (7.0,-3.0) node[anchor=north west] {12311};

\draw [->,line width=1.25pt,color=red] (4.7,2.0) -- (2.7,3.9);
\draw [->,line width=1.25pt,color=red] (8.7,2.0) -- (6.7,3.9);
\draw [->,line width=1.25pt,color=red] (-1.3,2.0) -- (-1.3,3.9);

\draw (0.7,1.4)-- (3.7,-0.5);
\draw (2.7,1.4)-- (5.7,-0.5);
\draw (6.7,1.4)-- (9.7,-0.5);
\draw [->,line width=1.25pt,color=red] (-0.3,-0.5) -- (0.7,1.4);
\draw [->,line width=1.25pt,color=red] (1.7,-0.5) -- (2.7,1.4);
\draw [->,line width=1.25pt,color=red] (7.7,-0.5) -- (6.7,1.4);

\draw (-0.3,-1.1)-- (3.7,-3.0);
\draw (1.7,-1.1)-- (5.7,-3.0);
\draw (7.7,-1.1)-- (7.7,-3);
\draw [->,line width=1.25pt,color=red] (7.7,-3.0) -- (9.7,-1.1);
\draw [->,line width=1.25pt,color=red] (-0.3,-3.0) -- (-4.1,-1.1);
\draw [->,line width=1.25pt,color=red] (1.7,-3.0) -- (-2.3,-1.1);
\draw [->,line width=1.25pt,color=red] (3.7,-3.0) -- (3.7,-1.1);
\draw [->,line width=1.25pt,color=red] (5.7,-3.0) -- (5.7,-1.1);
\end{tikzpicture}
\end{center}
\caption{
The decomposition of the Stirling poset $\Pi(5,3)$ into Boolean algebras.
Again, the matched elements are indicated with arrows.
The weight of the poset
is $\qStirling{5}{3} = 1+2(1+q)+3(1+q)^2+q^2+3q^2(1+q)+q^4$.}
\label{figure_pi_5_3_decomposition} 
\end{figure}

\section{Homological $q=-1$ phenomenon}
\label{section_homological_Stembridge}

Stembridge's $q=-1$ phenomenon~\cite{Stembridge, Stembridge_canonical}
and the more general
cyclic sieving phenomenon of 
Reiner, Stanton and White~\cite{Reiner_Stanton_White}
count symmetry classes in 
combinatorial objects
by evaluating their $q$-generating series at
a primitive root of unity. Recently Hersh, Shareshian and Stanton~\cite{Hersh_Shareshian_Stanton} 
have given a homological interpretation of 
the $q=-1$ phenomenon by viewing it as an
Euler characteristic computation on 
a chain complex supported by a poset.
In the best scenario,
the homology is concentrated in dimensions
of the same parity and one can identify a homology basis.
For further information about
algebraic discrete Morse theory,
see~\cite{Jollenbeck_Welker, Kozlov_discrete_Morse, Skoldberg}.

We will see the graded poset
$\Pi(n,k)$ supports an algebraic complex
$({\mathcal C},\partial)$.
The aforementioned matching for 
$\Pi(n,k)$ (Theorem~\ref{theorem_second_kind_acyclic_matching})
is a discrete Morse matching 
for this complex
and the unmatched elements occur in even ranks of the poset.
Hence using standard discrete Morse theory~\cite{Forman}, 
we can give a basis
for the homology.

We now review the relevant background.
We follow~\cite{Hersh_Shareshian_Stanton} here.
See also~\cite{Jollenbeck_Welker, Skoldberg}.
Let $P$ be a graded poset and $W_i$ denote the
rank $i$ elements.
We say the {\em poset $P$ supports a chain complex
$({\cal C}, \partial)$} of $\Fff$-vector spaces
$C_i$ if each $C_i$ has basis indexed by the rank $i$ elements
$W_i$
and 
$\partial_i: W_i \rightarrow W_{i-1}$
is a boundary map.
Furthermore, for
$x \in W_{i}$ and $y \in W_{i-1}$ the coefficient 
$\partial_{x,y}$ of $y$ in $\partial_i(x)$ is zero unless
$y <_P x$.

For
$w \in \Pi(n,k)$, let
\begin{align*}
     E(w) = \{i : w_i \mbox{ is even}
             \mbox{ and } w_j = w_i \mbox{ for some } j < i\}
\end{align*}
be the set
of all indices
of repeated even entries in the word $w$.
Define the boundary map $\partial$ on the elements of
$\Pi(n,k)$ by
\begin{equation}
\label{equation_boundary_map}
\partial(w) = 
              \sum_{j=1}^r (-1)^{j-1} \cdot w_1 \cdots w_{i_j - 1}
                    \cdot (w_{i_j} - 1) \cdot w_{i_j + 1} 
                    \cdots w_n,
\end{equation}
where
$E(w) = \{ i_1 < i_2 < \cdots < i_r\}$.
For example, 
if $w  = 122344$ then $E(122344) = \{3,6\}$
and
$\partial(122344) = 121344 - 122343$.
With this definition of the boundary operator
$\partial$,  we have the following lemma.

\begin{lemma}
The map $\partial$ is a boundary map on the algebraic complex
$({\mathcal C}, \partial)$ with the poset
$\Pi(n,k)$ as support.
\end{lemma}
\begin{proof}
By definition of $\partial$, we have
\begin{eqnarray*}
\partial^2(w) &=&
                 \sum_{{i_r} < {i_j} } (-1)^{j-1} \cdot
                                       (-1)^{r-1} \cdot
                     w_1 w_2 \cdots w_{i_r -1} 
                     \cdots
                     (w_{i_j} -1)  \cdots w_n
\\
             & & \:\: +  \sum_{{i_r} > {i_j} } (-1)^{j-1} \cdot
                                       (-1)^{r-2} \cdot
                     w_1 w_2 \cdots w_{i_j -1} 
                     \cdots
                     (w_{i_r} -1)  \cdots w_n,
\end{eqnarray*}
where the sum is over indices $i_r$ and $i_j$ with
$w_{i_j}, w_{i_r} \in E(w)$.
These two summations cancel since after switching 
$r$ and $j$ in the second summation,
the resulting expression becomes
the negative of the first. Hence we have that $\partial^2(w) = 0$.
\end{proof}


We have shown the graded poset
$\Pi(n,k)$ supports an algebraic complex
$({\mathcal C},\partial)$.
We will need a lemma due to 
Hersh, Shareshian and Stanton~\cite[Lemma 3.2]{Hersh_Shareshian_Stanton}.
This is part (ii) of the original statement of the lemma.

\begin{lemma}
[Hersh--Shareshian--Stanton]
\label{lemma_homology}
Let $P$ be a graded poset supporting an algebraic complex 
$(\mathcal{C}, \partial)$. 
Assume the poset $P$ has a Morse matching $M$ such that 
for all 
matched pairs $(y,x)$  with $y \prec x$ 
one has $\partial_{y,x} \in \Fff^*$.
If all unmatched poset elements occur in ranks of
the same parity, then $\dim ( H_i(\mathcal{C}, \partial)) = |P_i^{\un M}|$, 
that is,
the number of unmatched elements of rank $i$.
\end{lemma}

We can now state our result.

\begin{theorem}
For the algebraic complex $({\mathcal C}, \partial)$ 
supported by the Stirling poset of the second kind $\Pi(n,k)$,
a basis for the integer homology is given by the
weakly increasing
allowable $RG$-words in $\Allow(n,k)$.
Furthermore, 
we have
$$
   \sum_{i \geq 0} \dim (H_i(\mathcal{C}, \partial; \Zzz)) \cdot q^i =
     \qchoose{n-1 - \lfloor \frac{k}{2} \rfloor}
             {\lfloor \frac{k-1}{2} \rfloor}{q^2}.
$$
\label{theorem_homology_basis_Stirling_second}
\end{theorem}
\begin{proof}
By definition of 
the boundary map $\partial$, if 
$(x,y) \in M$
then $\partial_{y,x} = 1$
and all of the unmatched words in 
$\Pi(n,k)$ occur in even ranks. 
The conditions in Lemma~\ref{lemma_homology}
are satisfied. 
So $\sum_{i \geq 0} \dim (H_i(\mathcal{C}, \partial; \Zzz)) \cdot q^i$ is the
$q^2$-binomial coefficient 
in Lemma~\ref{lemma_generating_function}.
\end{proof}

\begin{remark}
{\rm
(A second proof of
Theorem~\ref{theorem_homology_basis_Stirling_second}.)
Theorem~\ref{theorem_homology_basis_Stirling_second}
can be proved without resorting to 
Lemma~\ref{lemma_homology} as follows.
The boundary map $\partial$
is supported on the Boolean algebras in the poset
decomposition
given in Theorem~\ref{theorem_Boolean_algebra_decomposition}.
Furthermore, the restriction
to one of these Boolean algebras is the natural boundary map 
on that Boolean algebra.  Hence the algebraic complex 
is a direct sum of 
algebraic complexes of Boolean algebras.  The only
summands that contribute any homology 
is the rank $0$ Boolean algebras, that is,
the unmatched elements.
}
\end{remark}

\section{$q$-Stirling numbers of the first kind}
\label{section_first_kind}

The (unsigned) {\em $q$-Stirling numbers of the first kind}
are defined by the recurrence formula
\begin{equation}
\label{equation_Stirling_one_recurrence}
     c_q[n,k] = c_q[n-1, k-1] + [n-1]_q \cdot c_q[n-1,k],
\end{equation}
where $c_q[n, 0] = \delta_{n,0}$.
When $q=1$, the Stirling number of the first kind
$c(n,k)$ enumerates permutations in the symmetric group
$\Sym_n$ having exactly $k$ disjoint cycles.
A combinatorial way to express $q$-Stirling
numbers of the first kind  is via rook placements;
see de~M\'edicis and Leroux~\cite{de_Medicis_Leroux_unified}.
Throughout a staircase chessboard of length~$m$
is a board 
with $m-i$ squares in the $i$th row
for $i = 1, \ldots, m-1$
and each row of squares is left-justified.

\begin{definition}
Let $\rook{m}{n}$ be the set of all ways
to place $n$ rooks onto a staircase chessboard
of length~$m$ so that no two rooks are in the same column.
For any rook placement $T \in \rook{m}{n}$, denote by $\below(T)$ the number of squares to the south of the rooks in $T$.
\end{definition}

\begin{theorem}
[de~M\'edicis--Leroux]
The $q$-Stirling number of the first kind 
$c_q[n,k]$ is given by
$$
     c_q[n,k] = \sum_{T \in \rook{n}{n-k}} q^{\below(T)},
$$
where the sum is over all rook placements of $n-k$ rooks
on a staircase board of length $n$.
\end{theorem}

We now define a subset $\alrook{n}{n-k}$ 
of rook placements in $\rook{n}{n-k}$ 
so that the $q$-Stirling number of the first kind $c_q[n,k]$ 
can be expressed as a statistic on the subset involving
$q$ and $1+q$.
The key is given any staircase chessboard, 
assign it a certain alternating shaded pattern.

\begin{definition}
Given any staircase chessboard, assign it a chequered pattern
such that every other antidiagonal strip of
squares is shaded, beginning with the lowest antidiagonal.
Let
$$
     \alrook{m}{n} = \{T \in \rook{m}{n} :
                      \mbox{all rooks are placed in shaded squares} \}
$$
For any rook placement $T \in \alrook{m}{n}$,
let $\nrow(T)$ denote the number of rooks in $T$
that are not in the first row.
Define
the weight to be
$\wt(T) = q^{\below(T)} \cdot (1+q)^{\nrow(T)}$.
\end{definition}

\begin{theorem}
\label{theorem_q_Stirling_first}
The $q$-Stirling number of the first kind is given by
$$
     \displaystyle c_q[n,k] = \sum_{T\in\alrook{n}{n-k}} \wt(T)
                      = \sum_{T\in\alrook{n}{n-k}} q^{\below(T)} 
                             \cdot (1+q)^{\nrow(T)},
$$
where the sum is over all rook placements of $n-k$ rooks
on an alternating shaded
staircase board of length $n$.
\end{theorem}

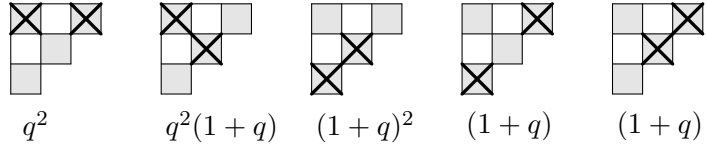
\begin{figure}
\centering
\begin{tikzpicture}
[line cap=round,line join=round,>=triangle 45,x=1.0cm,y=1.0cm, scale = 0.4]
\clip(-2,-1.5) rectangle (25,4.1);


\fill[color=gray,fill=gray,fill opacity=.2] (0,4) -- (0,3) -- (1,3) -- (1,4) -- cycle;
\fill[color=gray,fill=gray,fill opacity=.2] (2,4) -- (2,3) -- (3,3) -- (3,4) -- cycle;
\fill[color=gray,fill=gray,fill opacity=.2] (1,3) -- (1,2) -- (2,2) -- (2,3) -- cycle;
\fill[color=gray,fill=gray,fill opacity=.2] (0,2) -- (0,1) -- (1,1) -- (1,2) -- cycle;

\draw (0,4)-- (3,4);
\draw (0,3)-- (3,3);
\draw (0,2)-- (2,2);
\draw (0,1)-- (1,1);
\draw (0,4)-- (0,1);
\draw (1,4)-- (1,1);
\draw (2,4)-- (2,2);
\draw (3,4)-- (3,3);

\draw [line width=1.25pt] (0,4)-- (1,3);
\draw [line width=1.25pt] (1,4)-- (0,3);
\draw [line width=1.25pt] (2,4)-- (3,3);
\draw [line width=1.25pt] (3,4)-- (2,3);

\fill[color=gray,fill=gray,fill opacity=0.2] (5,4) -- (5,3) -- (6,3) -- (6,4) -- cycle;
\fill[color=gray,fill=gray,fill opacity=0.2] (7,4) -- (7,3) -- (8,3) -- (8,4) -- cycle;
\fill[color=gray,fill=gray,fill opacity=0.2] (5,2) -- (5,1) -- (6,1) -- (6,2) -- cycle;
\fill[color=gray,fill=gray,fill opacity=0.2] (6,3) -- (6,2) -- (7,2) -- (7,3) -- cycle;

\draw (5,4)-- (8,4);
\draw (5,3)-- (8,3);
\draw (5,2)-- (7,2);
\draw (5,1)-- (6,1);
\draw (5,4)-- (5,1);
\draw (6,4)-- (6,1);
\draw (7,4)-- (7,2);
\draw (8,4)-- (8,3);

\draw [line width=1.25pt] (5,4)-- (6,3);
\draw [line width=1.25pt] (6,4)-- (5,3);
\draw [line width=1.25pt] (6,3)-- (7,2);
\draw [line width=1.25pt] (7,3)-- (6,2);


\fill[color=gray,fill=gray,fill opacity=0.2] (10,4) -- (10,3) -- (11,3) -- (11,4) -- cycle;
\fill[color=gray,fill=gray,fill opacity=0.2] (12,4) -- (12,3) -- (13,3) -- (13,4) -- cycle;
\fill[color=gray,fill=gray,fill opacity=0.2] (10,2) -- (10,1) -- (11,1) -- (11,2) -- cycle;
\fill[color=gray,fill=gray,fill opacity=0.2] (11,3) -- (11,2) -- (12,2) -- (12,3) -- cycle;

\draw (10,4)-- (13,4);
\draw (10,3)-- (13,3);
\draw (10,2)-- (12,2);
\draw (10,1)-- (11,1);
\draw (10,4)-- (10,1);
\draw (11,4)-- (11,1);
\draw (12,4)-- (12,2);
\draw (13,4)-- (13,3);

\draw [line width=1.25pt] (10,2)-- (11,1);
\draw [line width=1.25pt] (11,2)-- (10,1);
\draw [line width=1.25pt] (11,3)-- (12,2);
\draw [line width=1.25pt] (12,3)-- (11,2);


\fill[color=gray,fill=gray,fill opacity=0.2] (15,4) -- (15,3) -- (16,3) -- (16,4) -- cycle;
\fill[color=gray,fill=gray,fill opacity=0.2] (17,4) -- (17,3) -- (18,3) -- (18,4) -- cycle;
\fill[color=gray,fill=gray,fill opacity=0.2] (15,2) -- (15,1) -- (16,1) -- (16,2) -- cycle;
\fill[color=gray,fill=gray,fill opacity=0.2] (16,3) -- (16,2) -- (17,2) -- (17,3) -- cycle;

\draw (15,4)-- (18,4);
\draw (15,3)-- (18,3);
\draw (15,2)-- (17,2);
\draw (15,1)-- (16,1);
\draw (15,4)-- (15,1);
\draw (16,4)-- (16,1);
\draw (17,4)-- (17,2);
\draw (18,4)-- (18,3);

\draw [line width=1.25pt] (15,2)-- (16,1);
\draw [line width=1.25pt] (16,2)-- (15,1);
\draw [line width=1.25pt] (17,4)-- (18,3);
\draw [line width=1.25pt] (18,4)-- (17,3);


\fill[color=gray,fill=gray,fill opacity=0.2] (20,4) -- (20,3) -- (21,3) -- (21,4) -- cycle;
\fill[color=gray,fill=gray,fill opacity=0.2] (22,4) -- (22,3) -- (23,3) -- (23,4) -- cycle;
\fill[color=gray,fill=gray,fill opacity=0.2] (20,2) -- (20,1) -- (21,1) -- (21,2) -- cycle;
\fill[color=gray,fill=gray,fill opacity=0.2] (21,3) -- (21,2) -- (22,2) -- (22,3) -- cycle;

\draw (20,4)-- (23,4);
\draw (20,3)-- (23,3);
\draw (20,2)-- (22,2);
\draw (20,1)-- (21,1);
\draw (20,4)-- (20,1);
\draw (21,4)-- (21,1);
\draw (22,4)-- (22,2);
\draw (23,4)-- (23,3);

\draw [line width=1.25pt] (21,3)-- (22,2);
\draw [line width=1.25pt] (22,3)-- (21,2);
\draw [line width=1.25pt] (22,4)-- (23,3);
\draw [line width=1.25pt] (23,4)-- (22,3);

\draw (0.05,0.77) node[anchor=north west] {$q^2$};
\draw (4.8,0.77) node[anchor=north west] {$q^2(1+q)$};
\draw (9.8,0.77) node[anchor=north west] {$(1+q)^2$};
\draw (14.8,0.77) node[anchor=north west] {$(1+q)$};
\draw (19.8,0.77) node[anchor=north west] {$(1+q)$};

\end{tikzpicture}
\caption{Computing the $q$-Stirling number
of the first kind $c_q[4,2]$
using $\alrook{4}{2}$.}
\label{allow_rooks_3_2}
\end{figure}

\begin{proof}
We proceed by  induction on $n$.
It is straightforward to see the result holds
for  $n=k=0$.
Suppose the result is true for alternating shaded staircase boards
of length $n-1$.
Then we have
\begin{eqnarray*}
\sum_{T\in\alrook{n}{n-k}} \wt(T) &=&
\sum_{\substack{T\in\alrook{n}{n-k}\\
     \text{leftmost column is empty}}} \wt(T)
+
\sum_{\substack{T\in\alrook{n}{n-k}\\
     \text{leftmost column is not empty}}} \wt(T) \\
		& = &
                      \sum_{T \in \alrook{n-1}{n-k}} \wt(T)
                      +
                      \sum_{T \in \alrook{n-1}{n-k-1}}
                      [n-1]_q \cdot \wt(T) \\ \vspace*{2mm}
		& = &  c_q[n-1, k-1] + [n-1]_q  \cdot c_q[n-1, k]\\
		& = &  c_q[n, k].
\end{eqnarray*}
In the second equality, the first term follows from the fact that
one can remove the leftmost column from the board, leaving a rook
placement of $n-k$ rooks on a length $n-1$ shaded board.
For the second term, we first consider where the rook occurs
in the leftmost column.
If the rook  occurs in the
$(2i + 1)$st entry from the bottom of the leftmost column,
where $0 \leq i < \lfloor (n-1)\slash 2 \rfloor$,
it contributes a weight
of $q^{2i} \cdot (1+q)$ since there are $2i$ squares below it
and the rook does not occur in the first row.
The only way a rook in the first column can
also occur in the
first row of a shaded staircase board is if 
the leftmost column has an odd number of squares, that is,
$n$ is even.
In this case
the rook would contribute a weight of $q^{n-2}$.
For $n$ even the overall weight contribution from a rook
in the first column is
$1 \cdot (1+q) + q^2 \cdot (1+q) + \cdots + q^{n-4} \cdot (1+q) + q^{n-2} = [n-1]_q$
and for $n$ odd the weight contribution is
$1 \cdot (1+q) + q^2 \cdot (1+q) + \cdots + q^{n-3} \cdot (1+q)  = [n-1]_q$.
Hence removing the first column from the staircase board 
along with
the rook that occurs in it leaves a shaded staircase board
of length $n-1$ with $n-k-1$ rooks.
The total weight 
lost 
is  $[n-1]_q$.
Finally, the last equality
is recurrence~(\ref{equation_Stirling_one_recurrence}).
\end{proof}

See Figure~\ref{allow_rooks_3_2} for the computation of
$c_q[4, 2]$ using allowable rook placements on length $4$
shaded staircase boards.

When we substitute $q = -1$ into
the $q$-Stirling number of the first kind,
the weight $\wt(T)$ of a rook placement $T$
will be $0$ if there is a rook in $T$ that is not in the first row.
Hence the Stirling number of the first kind $c_q[n,k]$
evaluated at $q = -1$ counts the number of rook placements
in $\alrook{n}{n-k}$ such that all of the rooks
occur in shaded squares of the first row.

\begin{corollary}
\label{corollary_shaded_rook_placements}
The $q$-Stirling number of the first kind
$c_q[n,k]$ evaluated at $q=-1$ gives the number of rook placements
in $\alrook{n}{n-k}$ where all of the rooks
occur in shaded squares in the first row, that is,
$$
     c_q[n,k] \big| _{q = -1} = \binom{\lfloor n/2 \rfloor}{n-k}.
$$
\end{corollary}

Let $d(n,k) = |\alrook{n}{n-k}|$.  We call
$d(n,k)$ 
the {\em allowable Stirling number of the first kind}.
See Table~\ref{table_allowable_Stirling_first} for values.

\begin{proposition}
The allowable Stirling numbers of the first kind $d(n,k)$ satisfy
the recurrence
$$
     d(n,k) = d(n-1, k-1)
                   + \left\lceil \frac{n-1}{2} \right\rceil \cdot d(n-1, k)
$$
with boundary conditions
$d(n,0) = \delta_{n,0}$,
$d(n,n) = 1$ for $n \geq 0$
 and 
$d(n,k) = 0$ when $k > n$.
\end{proposition}

\begin{proof}
For each $T\in \alrook{n}{n-k}$, 
there are two cases. 
If the leftmost column in $T$ is empty, 
then after deleting this column we obtain 
an allowable rook placement $T' \in \alrook{n-1}{n-k}$. 
Otherwise assume there is a rook in the leftmost column. 
We can first build an allowable rook placement 
$T' \in \alrook{n-1}{n-k-1}$ 
and then add a column of  $n-1$ squares 
with a rook in it to the left of $T'$ 
to form a rook placement in $\alrook{n}{n-k}$. 
Notice that the rook in the leftmost column 
can be only put into a shaded square, so 
there are $\lceil (n-1)\slash 2 \rceil$ 
possible squares to place the rook. 
Overall this case gives 
$\left\lceil (n-1)\slash 2 \right\rceil \cdot d(n-1, k)$ possibilities.
\end{proof}

\begin{table}
\begin{center}
\resizebox{6in}{!}
{
\begin{tabular}{|r|rrrrrrrrrrr|r|r|}
\hline
$n\backslash k$   &0  &1  &2  &3  &4  &5  &6  &7  &8  &9  &10 
&$r(n)$
&$n!$\\
\hline
0   &1  &&&&&&&&&&  &1 &1\\
1   &0  &1  &&&&&&&&&   &1 &1\\
2   &0  &1  &1  &&&&&&&&    &2 &2\\
3   &0  &1  &2  &1  &&&&&&&     &4 &6\\
4   &0  &2  &5  &4  &1  &&&&&&      &12 &24 \\
5   &0  &4  &12  &13 &6  &1  &&&&&       &36 &120 \\
6   &0  &12  &40  &51 &31 &9  &1  &&&&        &144 &720\\
7   &0  &36  &132  &193 &144 &58 &12 &1  &&&         &576 &5040\\
8   &0  &144  &564  &904    &769    &376    &106 &16 &1  &&  &2880 &40320 \\
9   &0  &576  &2400  &4180    &3980    &2273    &800    &170    &20 &1  &   &14400 &362880\\
10  &0  &2880  &12576  &23300    &24080    &15345   &6273  &1650   &270    &25 &1  &86400
&3628800\\
\hline
\end{tabular}
}
\caption{The allowable Stirling numbers of the first kind
$d(n,k)$, their row sum
$r(n)$ and $n!$ for $0 \leq n \leq 10$.}
\label{table_allowable_Stirling_first}
\end{center}
\end{table}

Certain allowable Stirling numbers of the first kind 
have closed forms as follows.
Here we let
$r(n) =  \sum_{k=0}^n d(n,k)$
denote the row sum of the allowable Stirling numbers
of the first kind.

\begin{proposition}
The allowable Stirling numbers of the first kind satisfy
\begin{align}
\label{equation_stir1_1}
     d(n, 1) &=
         \begin{cases}
             \left( \frac{n-1}{2} \right)!^2
             &\text{    for $n$ odd},\\
             \frac{n}{2} \cdot \left( \frac{n-1}{2} \right)!^2
             &\text{    for $n$ even,}
         \end{cases}\\
\label{equation_stir1_n}
	d(n, n-1) &= \Big\lfloor \frac{n}{2} \Big\rfloor \cdot \Big\lceil \frac{n}{2} \Big\rceil ,\\
\label{equation_stir1_sum}
	r(n) &= d(n+2,1).
\end{align}
\end{proposition}

\begin{proof}
We first prove~(\ref{equation_stir1_sum}).
Let
$T \in \alrook{n+2}{1}$ be a rook placement
on a shaded board.
Since rooks are only allowed to be placed in shaded
squares, the two rooks in the rightmost two columns must be in the
bottommost antidiagonal. Delete the
two longest anti-diagonals from $T$ to
obtain~$T'$. Since the shaded squares are preserved, $T'$ is still
allowable with the longest column length $n$. The rightmost two
rooks in $T$ are deleted to form $T'$, 
giving at most $n-1$ rooks in
$T'$. Hence $d(n+2,1) \leq r(n)$.

On the other hand, for any rook placement $T$ 
with at most $n-1$ rooks on a shaded
staircase board of length $n$, we can 
add two anti-diagonals to $T$ and place a rook in the bottom row for
each empty column in the new chessboard to obtain $T'$. 
The board $T'$
has $n+1$ rooks and $n+1$ columns, hence $r(n) \leq d(n+2,1)$. 
Hence we have the equality~(\ref{equation_stir1_sum}).

The expression $d(n, n-1)$
counts the number of rook placements of length $n$ 
using $1$ rook.
This is the same as counting the number of shaded squares in a 
length $n$
staircase chessboard.  Counting column by column,
beginning from the right, gives
 $1+1+2+2+ \cdots + \lfloor
n \slash 2 \rfloor = \lfloor n \slash 2 \rfloor \cdot \lceil n \slash
2 \rceil$.

Finally,
the expression $d(n,1)$
counts the number of rook
placements with $n-1$ columns and $n-1$ rooks. Thus each column must
have a rook. 
For each column with $k$ squares,
there are $\lceil k/2 \rceil$
shaded squares, hence $\lceil k \slash 2 \rceil$
choices for the rook. This gives $\left( (n-1) \slash 2 \right)!^2$
ways when $n$ is odd and $(n \slash 2) \cdot \left( (n-1)\slash
2 \right)!^2$ ways when $n$ is even.
\end{proof}

\section{Structure and topology of the Stirling poset of the first kind}
\label{section_Stirling_first_poset}

We define a poset structure on rook placements on a staircase shape
board.  
For rook placements
 $T$ and $T'$ in $\rook{m}{n}$, let $T\prec T'$ if $T'$
can be obtained from $T$ by either moving a rook to the left (west) or up
(north) by one square. 
We call this poset 
{\em the Stirling poset of the first kind} 
and denote it by $\prook{m}{n}$.
It is straightforward to check
that
the poset $\prook{m}{n}$ is graded of rank 
$(m-2) + (m-3) \cdots + (m-n-1) = (m-1) \cdot n - \binom{n+1}{2}$ and its rank generating function is $c_q[m, m-n]$. 
See Figure~\ref{rook_poset_2} for an example.

We wish to
study the topological properties of the Stirling poset of the
first kind.
To do so,
we define a matching $M$ on the poset
as follows.
Given any rook placement $T \in \prook{m}{n}$, 
let $r$ be the first rook (reading from left to right) 
that is not in a shaded square of the first row.
Match $T$ to $T'$ where $T'$ is obtained from $T$ by moving 
the rook $r$ one square down if $r$ is not in a shaded square, 
or one square up if $r$ is in a shaded square but not in the first row. 
It is straightforward to check that the unmatched rook placements 
are the ones where all of the rooks occur
in the shaded squares of the first row.

As an example, 
the matching for $\prook{4}{2}$ is shown in 
Figure~\ref{rook_poset_2}, where an upward arrow indicates a matching 
and other edges indicate 
the remaining cover relations.
Observe
the unmatched rook placements are the ones with all the rooks
occurring in the shaded squares in the first row. By the way a
chessboard is shaded, the unmatched rook placements only appear in
even ranks in the poset.

\begin{figure}
\centering
\begin{tikzpicture}[>=triangle 45, x=1.0cm,y=1.0cm, scale = 0.3]
\clip(-4,-15.5) rectangle (16,9.5);
\fill[color=gray,fill=gray,fill opacity=0.2] (7,9) -- (7,8) -- (8,8) -- (8,9) -- cycle;
\fill[color=gray,fill=gray,fill opacity=0.2] (5,9) -- (5,8) -- (6,8) -- (6,9) -- cycle;
\fill[color=gray,fill=gray,fill opacity=0.2] (6,8) -- (6,7) -- (7,7) -- (7,8) -- cycle;
\fill[color=gray,fill=gray,fill opacity=0.2] (5,7) -- (5,6) -- (6,6) -- (6,7) -- cycle;
\fill[color=gray,fill=gray,fill opacity=0.2] (12,2) -- (12,1) -- (13,1) -- (13,2) -- cycle;
\fill[color=gray,fill=gray,fill opacity=0.2] (10,2) -- (10,1) -- (11,1) -- (11,2) -- cycle;
\fill[color=gray,fill=gray,fill opacity=0.2] (11,1) -- (11,0) -- (12,0) -- (12,1) -- cycle;
\fill[color=gray,fill=gray,fill opacity=0.2] (10,0) -- (10,-1) -- (11,-1) -- (11,0) -- cycle;
\fill[color=gray,fill=gray,fill opacity=0.2] (7,2) -- (7,1) -- (8,1) -- (8,2) -- cycle;
\fill[color=gray,fill=gray,fill opacity=0.2] (5,2) -- (5,1) -- (6,1) -- (6,2) -- cycle;
\fill[color=gray,fill=gray,fill opacity=0.2] (6,1) -- (6,0) -- (7,0) -- (7,1) -- cycle;
\fill[color=gray,fill=gray,fill opacity=0.2] (5,0) -- (5,-1) -- (6,-1) -- (6,0) -- cycle;
\fill[color=gray,fill=gray,fill opacity=0.2] (2,2) -- (2,1) -- (3,1) -- (3,2) -- cycle;
\fill[color=gray,fill=gray,fill opacity=0.2] (0,2) -- (0,1) -- (1,1) -- (1,2) -- cycle;
\fill[color=gray,fill=gray,fill opacity=0.2] (1,1) -- (1,0) -- (2,0) -- (2,1) -- cycle;
\fill[color=gray,fill=gray,fill opacity=0.2] (0,0) -- (0,-1) -- (1,-1) -- (1,0) -- cycle;
\fill[color=gray,fill=gray,fill opacity=0.2] (-1,-5) -- (-1,-6) -- (0,-6) -- (0,-5) -- cycle;
\fill[color=gray,fill=gray,fill opacity=0.2] (-3,-5) -- (-3,-6) -- (-2,-6) -- (-2,-5) -- cycle;
\fill[color=gray,fill=gray,fill opacity=0.2] (-2,-6) -- (-2,-7) -- (-1,-7) -- (-1,-6) -- cycle;
\fill[color=gray,fill=gray,fill opacity=0.2] (-3,-7) -- (-3,-8) -- (-2,-8) -- (-2,-7) -- cycle;
\fill[color=gray,fill=gray,fill opacity=0.2] (2,-5) -- (2,-6) -- (3,-6) -- (3,-5) -- cycle;
\fill[color=gray,fill=gray,fill opacity=0.2] (2,-7) -- (2,-8) -- (3,-8) -- (3,-7) -- cycle;
\fill[color=gray,fill=gray,fill opacity=0.2] (3,-6) -- (3,-7) -- (4,-7) -- (4,-6) -- cycle;
\fill[color=gray,fill=gray,fill opacity=0.2] (4,-5) -- (4,-6) -- (5,-6) -- (5,-5) -- cycle;
\fill[color=gray,fill=gray,fill opacity=0.2] (7,-5) -- (7,-6) -- (8,-6) -- (8,-5) -- cycle;
\fill[color=gray,fill=gray,fill opacity=0.2] (7,-7) -- (7,-8) -- (8,-8) -- (8,-7) -- cycle;
\fill[color=gray,fill=gray,fill opacity=0.2] (8,-6) -- (8,-7) -- (9,-7) -- (9,-6) -- cycle;
\fill[color=gray,fill=gray,fill opacity=0.2] (9,-5) -- (9,-6) -- (10,-6) -- (10,-5) -- cycle;
\fill[color=gray,fill=gray,fill opacity=0.2] (12,-5) -- (12,-6) -- (13,-6) -- (13,-5) -- cycle;
\fill[color=gray,fill=gray,fill opacity=0.2] (14,-5) -- (14,-6) -- (15,-6) -- (15,-5) -- cycle;
\fill[color=gray,fill=gray,fill opacity=0.2] (13,-6) -- (13,-7) -- (14,-7) -- (14,-6) -- cycle;
\fill[color=gray,fill=gray,fill opacity=0.2] (12,-7) -- (12,-8) -- (13,-8) -- (13,-7) -- cycle;
\fill[color=gray,fill=gray,fill opacity=0.2] (12,-12) -- (12,-13) -- (13,-13) -- (13,-12) -- cycle;
\fill[color=gray,fill=gray,fill opacity=0.2] (11,-13) -- (11,-14) -- (12,-14) -- (12,-13) -- cycle;
\fill[color=gray,fill=gray,fill opacity=0.2] (10,-12) -- (10,-13) -- (11,-13) -- (11,-12) -- cycle;
\fill[color=gray,fill=gray,fill opacity=0.2] (10,-14) -- (10,-15) -- (11,-15) -- (11,-14) -- cycle;
\fill[color=gray,fill=gray,fill opacity=0.2] (7,-12) -- (7,-13) -- (8,-13) -- (8,-12) -- cycle;
\fill[color=gray,fill=gray,fill opacity=0.2] (6,-13) -- (6,-14) -- (7,-14) -- (7,-13) -- cycle;
\fill[color=gray,fill=gray,fill opacity=0.2] (5,-14) -- (5,-15) -- (6,-15) -- (6,-14) -- cycle;
\fill[color=gray,fill=gray,fill opacity=0.2] (5,-12) -- (5,-13) -- (6,-13) -- (6,-12) -- cycle;
\fill[color=gray,fill=gray,fill opacity=0.2] (2,-12) -- (2,-13) -- (3,-13) -- (3,-12) -- cycle;
\fill[color=gray,fill=gray,fill opacity=0.2] (1,-13) -- (1,-14) -- (2,-14) -- (2,-13) -- cycle;
\fill[color=gray,fill=gray,fill opacity=0.2] (0,-12) -- (0,-13) -- (1,-13) -- (1,-12) -- cycle;
\fill[color=gray,fill=gray,fill opacity=0.2] (0,-14) -- (0,-15) -- (1,-15) -- (1,-14) -- cycle;
\draw (5.0,9.0)-- (8.0,9.0);
\draw (5.0,8.0)-- (8.0,8.0);
\draw (5.0,7.0)-- (7.0,7.0);
\draw (5.0,6.0)-- (6.0,6.0);
\draw (5.0,9.0)-- (5.0,6.0);
\draw (6.0,9.0)-- (6.0,6.0);
\draw (7.0,9.0)-- (7.0,7.0);
\draw (8.0,9.0)-- (8.0,8.0);
\draw (5.0,2.0)-- (8.0,2.0);
\draw (5.0,1.0)-- (8.0,1.0);
\draw (5.0,0.0)-- (7.0,0.0);
\draw (5.0,-1.0)-- (6.0,-1.0);
\draw (5.0,2.0)-- (5.0,-1.0);
\draw (6.0,2.0)-- (6.0,-1.0);
\draw (7.0,2.0)-- (7.0,0.0);
\draw (8.0,2.0)-- (8.0,1.0);
\draw (10.0,2.0)-- (13.0,2.0);
\draw (10.0,1.0)-- (13.0,1.0);
\draw (10.0,0.0)-- (12.0,0.0);
\draw (10.0,-1.0)-- (11.0,-1.0);
\draw (10.0,2.0)-- (10.0,-1.0);
\draw (11.0,2.0)-- (11.0,-1.0);
\draw (12.0,2.0)-- (12.0,0.0);
\draw (13.0,2.0)-- (13.0,1.0);
\draw (0.0,2.0)-- (3.0,2.0);
\draw (0.0,1.0)-- (3.0,1.0);
\draw (0.0,0.0)-- (2.0,0.0);
\draw (0.0,-1.0)-- (1.0,-1.0);
\draw (0.0,2.0)-- (0.0,-1.0);
\draw (1.0,2.0)-- (1.0,-1.0);
\draw (2.0,2.0)-- (2.0,0.0);
\draw (3.0,2.0)-- (3.0,1.0);
\draw (12.0,-5.0)-- (15.0,-5.0);
\draw (12.0,-6.0)-- (15.0,-6.0);
\draw (12.0,-7.0)-- (14.0,-7.0);
\draw (12.0,-8.0)-- (13.0,-8.0);
\draw (12.0,-5.0)-- (12.0,-8.0);
\draw (13.0,-5.0)-- (13.0,-8.0);
\draw (14.0,-5.0)-- (14.0,-7.0);
\draw (15.0,-5.0)-- (15.0,-6.0);
\draw (7.0,-5.0)-- (10.0,-5.0);
\draw (7.0,-6.0)-- (10.0,-6.0);
\draw (7.0,-7.0)-- (9.0,-7.0);
\draw (7.0,-8.0)-- (8.0,-8.0);
\draw (7.0,-5.0)-- (7.0,-8.0);
\draw (8.0,-5.0)-- (8.0,-8.0);
\draw (9.0,-5.0)-- (9.0,-7.0);
\draw (10.0,-5.0)-- (10.0,-6.0);
\draw (2.0,-5.0)-- (5.0,-5.0);
\draw (2.0,-6.0)-- (5.0,-6.0);
\draw (2.0,-7.0)-- (4.0,-7.0);
\draw (2.0,-8.0)-- (3.0,-8.0);
\draw (2.0,-5.0)-- (2.0,-8.0);
\draw (3.0,-5.0)-- (3.0,-8.0);
\draw (4.0,-5.0)-- (4.0,-7.0);
\draw (5.0,-5.0)-- (5.0,-6.0);
\draw (-3.0,-5.0)-- (0.0,-5.0);
\draw (-3.0,-6.0)-- (0.0,-6.0);
\draw (-3.0,-7.0)-- (-1.0,-7.0);
\draw (-3.0,-8.0)-- (-2.0,-8.0);
\draw (-3.0,-5.0)-- (-3.0,-8.0);
\draw (-2.0,-5.0)-- (-2.0,-8.0);
\draw (-1.0,-5.0)-- (-1.0,-7.0);
\draw (0.0,-5.0)-- (0.0,-6.0);
\draw (0.0,-12.0)-- (3.0,-12.0);
\draw (0.0,-13.0)-- (3.0,-13.0);
\draw (0.0,-14.0)-- (2.0,-14.0);
\draw (0.0,-15.0)-- (1.0,-15.0);
\draw (0.0,-12.0)-- (0.0,-15.0);
\draw (1.0,-12.0)-- (1.0,-15.0);
\draw (2.0,-12.0)-- (2.0,-14.0);
\draw (3.0,-12.0)-- (3.0,-13.0);
\draw (5.0,-12.0)-- (8.0,-12.0);
\draw (5.0,-13.0)-- (8.0,-13.0);
\draw (5.0,-14.0)-- (7.0,-14.0);
\draw (5.0,-15.0)-- (6.0,-15.0);
\draw (5.0,-12.0)-- (5.0,-15.0);
\draw (6.0,-12.0)-- (6.0,-15.0);
\draw (7.0,-12.0)-- (7.0,-14.0);
\draw (8.0,-12.0)-- (8.0,-13.0);
\draw (10.0,-12.0)-- (13.0,-12.0);
\draw (10.0,-13.0)-- (13.0,-13.0);
\draw (10.0,-14.0)-- (12.0,-14.0);
\draw (10.0,-15.0)-- (11.0,-15.0);
\draw (10.0,-12.0)-- (10.0,-15.0);
\draw (11.0,-12.0)-- (11.0,-15.0);
\draw (12.0,-12.0)-- (12.0,-14.0);
\draw (13.0,-12.0)-- (13.0,-13.0);
\draw (5.0,9.0)-- (6.0,8.0);
\draw (6.0,9.0)-- (5.0,8.0);
\draw (6.0,9.0)-- (7.0,8.0);
\draw (7.0,9.0)-- (6.0,8.0);
\draw (0.0,2.0)-- (1.0,1.0);
\draw (1.0,2.0)-- (0.0,1.0);
\draw (2.0,2.0)-- (3.0,1.0);
\draw (3.0,2.0)-- (2.0,1.0);
\draw (6.0,2.0)-- (7.0,1.0);
\draw (7.0,2.0)-- (6.0,1.0);
\draw (5.0,1.0)-- (6.0,-0.0);
\draw (6.0,1.0)-- (5.0,0.0);
\draw (10.0,2.0)-- (11.0,1.0);
\draw (11.0,2.0)-- (10.0,1.0);
\draw (11.0,1.0)-- (12.0,0.0);
\draw (12.0,1.0)-- (11.0,-0.0);
\draw (-2.0,-5.0)-- (-1.0,-6.0);
\draw (-1.0,-5.0)-- (-2.0,-6.0);
\draw (-1.0,-5.0)-- (0.0,-6.0);
\draw (0.0,-5.0)-- (-1.0,-6.0);
\draw (2.0,-6.0)-- (3.0,-7.0);
\draw (3.0,-6.0)-- (2.0,-7.0);
\draw (4.0,-5.0)-- (5.0,-6.0);
\draw (5.0,-5.0)-- (4.0,-6.0);
\draw (8.0,-5.0)-- (9.0,-6.0);
\draw (9.0,-5.0)-- (8.0,-6.0);
\draw (7.0,-7.0)-- (8.0,-8.0);
\draw (8.0,-7.0)-- (7.0,-8.0);
\draw (12.0,-6.0)-- (13.0,-7.0);
\draw (13.0,-6.0)-- (12.0,-7.0);
\draw (13.0,-6.0)-- (14.0,-7.0);
\draw (14.0,-6.0)-- (13.0,-7.0);
\draw (1.0,-13.0)-- (2.0,-14.0);
\draw (2.0,-13.0)-- (1.0,-14.0);
\draw (2.0,-12.0)-- (3.0,-13.0);
\draw (3.0,-12.0)-- (2.0,-13.0);
\draw (5.0,-14.0)-- (6.0,-15.0);
\draw (6.0,-14.0)-- (5.0,-15.0);
\draw (10.0,-14.0)-- (11.0,-15.0);
\draw (11.0,-14.0)-- (10.0,-15.0);
\draw (7.0,-12.0)-- (8.0,-13.0);
\draw (8.0,-12.0)-- (7.0,-13.0);
\draw (11.0,-13.0)-- (12.0,-14.0);
\draw (12.0,-13.0)-- (11.0,-14.0);
\draw [->,line width=1.25pt,color=red] (8.0,-4.0) -- (6.0,-2.0);
\draw [->,line width=1.25pt,color=red] (1.0,-11.0) -- (-2.0,-9.0);
\draw [->,line width=1.25pt,color=red] (6.0,-11.0) -- (4.0,-9.0);
\draw (6.0,5.0)-- (2.0,3.0);
\draw (6.0,5.0)-- (6.0,3.0);
\draw [->,line width=1.25pt,color=red] (11.0,3.0) -- (6.0,5.0);
\draw (1.0,-2.0)-- (3.0,-4.0);
\draw (9.0,-9.0)-- (11.0,-11.0);
\draw (1.0,-2.0)-- (-2.0,-4.0);
\draw (6.0,-2.0)-- (13.0,-4.0);
\draw (6.0,-2.0)-- (3.0,-4.0);
\draw (11.0,-2.0)-- (13.0,-4.0);
\draw (4.0,-9.0)-- (1.0,-11.0);
\draw (9.0,-9.0)-- (6.0,-11.0);
\draw [->,line width=1.25pt,color=red] (11.0,-11.0) -- (13.0,-9.0);
\end{tikzpicture}
\caption{Example of $\prook{4}{2}$ with its matching.  
There is one unmatched rook placement in rank $2$.
The rank generating function of this poset is
$c_q[4,2] = 3 + 4q + 3q^2 + q^3$.}
\label{rook_poset_2}
\end{figure}
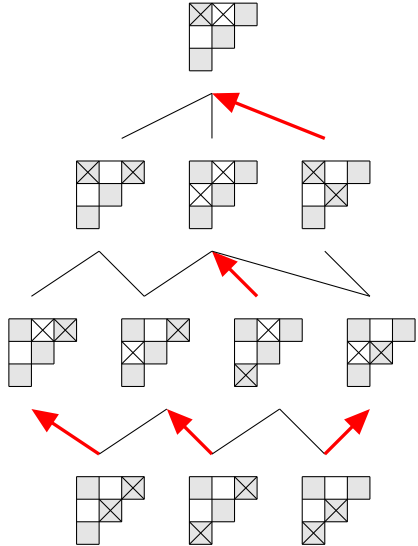

We have a $q$-analogue of Corollary~\ref{corollary_shaded_rook_placements}.

\begin{theorem}
\label{theorem_gen_stir1}
For the Stirling poset of the
first kind $\prook{m}{n}$
the generating function for the unmatched rook placements is
$$
\underset{\substack{T\in\prook{m}{n}\\ \text{$T$ unmatched}}} \sum\wt(T) = 
q^{n(n-1)} \cdot \qchoose{\lfloor\frac{m}{2}\rfloor}{n}{q^2}\,.
$$
\end{theorem}
\begin{proof}
The number of unmatched rook placements 
in rank $2j$ 
in the poset
$\prook{m}{n}$ 
is the same as the number of integer partitions 
$\lambda = (\lambda_1, \ldots, \lambda_n)$ of $2j$ 
into $n$ distinct non-negative even parts, 
with each $\lambda_i \leq m-1-(2i-1)$. 
Alternatively, this is the number of partitions 
$\delta = (\delta_1, \ldots, \delta_n)$ of 
$2j-(0+2+\cdots+(2n-2)) = 2j-n(n-1)$ into $n$ non-negative even parts, 
where 
each part  $\delta_i$
satisfies~$\delta_i  = \lambda_i - (2n-(2i-2)) \leq m-2n$
for $i = 1, \ldots, n$. Thus we have
\begin{eqnarray*}
\sum_{\substack{T\in\prook{m}{n}\\ 
          \text{$T$ unmatched}}}   \wt(T)
           &=& \sum_{j\geq 0} 
               \sum_{\substack{(\lambda_1, \ldots, \lambda_n) \vdash 2j \\ 
                               0 \leq \lambda_i \leq m-1-(2i-1) \\ 
                            \lambda_i \text{ distinct even integers} }} 
                               q^{|\lambda|} \\
           &=& q^{n(n-1)} \cdot \sum_{2j-n(n-1)\geq 0} 
               \sum_{\substack{\lambda \vdash 2j-n(n-1) \\ 
                     0\leq \lambda_i \leq m-2n \\ 
                     i = 1, \ldots, n \\ \lambda_i \text{ even integers}}} 
                     q^{|\lambda|}\\
     &=& q^{n(n-1)} \cdot \sum_{j-\frac{n(n-1)}{2} \geq 0} 
                    \sum_{\substack{\lambda \vdash j-\frac{n(n-1)}{2} \\ 
                     0 \leq \lambda_i \leq \lfloor\frac{m}{2}\rfloor-n \\ 
                     i = 1, \ldots, n}} (q^2)^{|\lambda|}\,.
\end{eqnarray*}

The last (double) sum is over all integer partitions into at most $n$ parts where each part is at most $\lfloor m \slash 2 \rfloor - n$. Hence this sum is given by the Gaussian polynomial $\qchoose{\lfloor m/2 \rfloor}{n}{q^2}$, proving the desired identity.
\end{proof}

Given a rook placement $T \in \rook{m}{n}$, 
we can associate to it a {\em rook word}
$w_T = w_1 w_2 \ldots w_{m-1}$ where~$w_i$ is one plus the number of squares below the column $i$ rook.
If column $i$ is empty, let $w_i = 0$.
See Figure~\ref{figure_rook_word}
for an example.

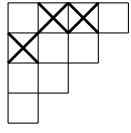
\begin{figure}
\centering
\begin{tikzpicture}[>=triangle 45, x=1.0cm,y=1.0cm, scale = 0.4]
\clip(-0.5,-0.5) rectangle (5,5);
\draw (0,4)-- (4,4);
\draw (0,3)-- (4,3);
\draw (0,2)-- (3,2);
\draw (0,1)-- (2,1);
\draw (0,4)-- (0,0);
\draw (0,0)-- (1,0);
\draw (1,0)-- (1,4);
\draw (2,4)-- (2,1);
\draw (3,2)-- (3,4);
\draw (4,4)-- (4,3);
\draw [line width=1.05pt] (1,3)-- (0,2);
\draw [line width=1.05pt] (0,3)-- (1,2);
\draw [line width=1.05pt] (1,4)-- (2,3);
\draw [line width=1.05pt] (2,4)-- (1,3);
\draw [line width=1.05pt] (2,4)-- (3,3);
\draw [line width=1.05pt] (3,4)-- (2,3);
\end{tikzpicture}
\caption{A rook placement $T$ with rook word $w_T = 3320$.}
\label{figure_rook_word}
\end{figure}

\begin{lemma}
\label{lemma_gradient_new_first}
Let $T$ and $T'$ be two distinct elements in the Stirling poset of the
first kind such that
$T \prec u(T) \succ T' \prec u(T')$
is a gradient path.
Then the rook words
satisfy the inequality  $w_T <_{\rm lex} w_{T'}$.
\end{lemma}
\begin{proof}
Let $w_T = w_1 \cdots w_n$.
Since $u(T)$ is obtained from $T$ by
shifting a rook $a$ in column $i$ up by one square,
we have $w_{u(T)} = w_1 \cdots (w_i + 1) \cdots w_n$.
By definition of the matching,
in the rook placement~$T$
the rook $a$ was in a shaded square not in
the first row.
In the rook placement $u(T)$ the rook $a$ is now in an unshaded
square. Furthermore, all of the rooks in the leftmost $i - 1$
columns of $T$ are in shaded squares in the first row.

The rook placement $T'$ is obtained from $u(T)$ by shifting a rook to
the right or down.
We first show that $T'$ cannot be obtained by shifting 
a rook in $u(T)$ down by one square.

Suppose a rook $b$ in column $j \neq i$ of $u(T)$ is shifted down
to form $T'$.
If $j < i$ since all of the rooks in columns $1$ through $i-1$
occur in shaded squares of the first row, the rook $b$ is now
in an unshaded square in the rook placement $T'$.
Hence if it is matched with another rook placement, it will
be of one rank lower, contradicting the fact that
we assumed $T'$ was part of a gradient path
$T \prec u(T) \succ T' \prec u(T')$. 
If $j > i$
then the rook $a$ in column $i$ of $T'$ 
is in an unshaded square and hence~$T'$ should be matched 
to a rook placement in one lower rank.  Again,
this contradicts our gradient path assumption. 
Hence this case cannot occur.

The remaining case is when
a rook in $u(T)$ occurring in the $j$th
column for some index $j < n$ is shifted to the right
to form $T'$.  Note this implies the $(j+1)$st column of
$T$ had no rooks in it.
If $j < i$, then since $b$ in column $j$ in $u(T)$ 
is in a shaded square of the first row, 
it is shifted to an unshaded square in $T'$ 
and hence $T'$ is matched to a rook placement in one lower rank. 
If $j > i$ then $a$ in $T'$ is the first rook that does not 
appear in a shaded square of the first row. 
Hence $T'$ is matched to some rook placement of one rank lower, 
contradicting the gradient path assumption.

The only remaining possibility is when $j=i$.
Then the rook $a$ in $u(T)$ is shifted to a shaded square in $T'$,
and hence 
$w_{T} = w_1 \cdots w_{i-1} \cdot w_i \cdot 0 \cdot w_{i+2} \cdots w_n 
>_{\rm lex} 
 w_1 \cdots w_{i-1} \cdot 0 \cdot (w_i - 1) \cdot w_{i+2} \cdots w_n = w_{T'}$,
as desired.
\end{proof}

\begin{theorem}
\label{theorem_acyclic_stir1}
The matching $M$ on 
the Stirling poset of the first kind
$\prook{m}{n}$ is an acyclic matching, that is,
the Stirling poset has a discrete Morse matching.
\end{theorem}
The proof is similar to that of
Theorem~\ref{theorem_second_kind_acyclic_matching}, and
thus omitted.

Next we give a decomposition of the Stirling poset of the first kind
$\Gamma(m,n)$ into Boolean algebras indexed by the allowable rook
placements.  This will lead to a boundary map on the algebraic
complex with $\Gamma(m,n)$ as the support.  For any
$T \in \alrook{m}{n}$, let $\alpha(T)$ be the rook placement obtained
by shifting every rook that is not in the first row up by one. Then we
have the following theorem.

\begin{theorem}
\label{theorem_decomposition_Stirling_poset_first}
The Stirling poset of the first kind $\Gamma(n,k)$ 
can be decomposed as disjoint union of Boolean intervals
$$
    \Gamma(m,n) = \underset{T \in \alrook{m}{n}}\disjointunion
                    [T, \alpha(T)] .
$$
Furthermore, if
$T \in \alrook{m}{n}$ has weight
$\wt(T) = q^i \cdot (1+q)^j$,
then
the rank of the element $T$ is $i$ and the interval
$[T, \alpha(T)]$ is isomorphic to the Boolean algebra on
$j$ elements.
\end{theorem}

\begin{proof}
We first show that for any $T \in \alrook{m}{n}$ 
with $\wt(T) = q^i \cdot (1+q)^j$
that the interval $[T, \alpha(T)] \cong B_j$. 
Since $\wt(T) = q^i \cdot (1+q)^j$, the rank of $T$ is $i$ 
and there are $j$ rooks in $T$ that are not in the first row. 
The rank $i+l$ elements in 
the interval $[T, \alpha(T)]$ correspond to shifting 
$l$ of those rooks up by one. 
It is straightforward to see that in the interval $[T, \alpha(T)]$ 
all of the elements except $T$ are in 
$\rook{m}{n} - \alrook{m}{n}$ 
since the rook that is shifted up by one will not be in a shaded square.

We next need to show that every element $
T \in \Gamma(m,n)$ occurs in some Boolean interval in this decomposition. 
This is vacuously true if $T \in \alrook{m}{n}$. 
Otherwise there are some rooks in $T$ that are not in shaded squares. 
Shift all such rooks down by one to obtain 
an allowable rook placement associated to $T$.
\end{proof}

Given a rook placement $T \in \prook{m}{n}$, 
let $N(T) = \{r_1, r_2, \ldots, r_s\}$ 
be the set of all rooks in $T$ that are not in shaded squares, 
where the rooks $r_i$ are labeled from left to right. 
We define the map $\partial$ as follows.

\begin{definition}
\label{definition_another_boundary}
Let $\partial: \prook{m}{n} \longrightarrow \mathbb{Z}[\prook{m}{n}]$ 
be the map defined by
$$
     \displaystyle 
     \partial(T) = \underset{r_i \in N(T)}\sum (-1)^{i-1} \cdot T_{r_i}\,,
$$
where $T_{r_i}$ is obtained by moving the rook $r_i$ in $T$ 
down by one square.
\end{definition}

\begin{lemma}
\label{lemma_boundary_stir1}
The map $\partial$ 
in Definition~\ref{definition_another_boundary}
 is a boundary map on the 
algebraic complex with $\prook{m}{n}$ as the support.
\end{lemma}
\begin{proof}
The boundary map $\partial$ is supported
on the Boolean algebra decomposition
of the Stirling poset of the first
kind appearing in 
Theorem~\ref{theorem_decomposition_Stirling_poset_first}.
The second proof of 
Theorem~\ref{theorem_homology_basis_Stirling_second}
applies again
to show $\partial$ is a boundary map.
\end{proof}

\begin{theorem}
\label{theorem_algebraic_complex_first_kind}
For the algebraic complex $({\mathcal C}, \partial)$ 
supported by the Stirling poset of the first kind 
$\prook{m}{n}$,
a basis for the integer homology is given by 
the rook placements in 
$\rook{m}{n}$ having all of the rooks occur in shaded
squares in the first row.  Furthermore,
$$
     \sum_{i \geq 0} \dim (H_i(\mathcal{C}, \partial; \Zzz)) 
           \cdot q^i = q^{n(n-1)} \cdot
           \qchoose{\lfloor\frac{m}{2}\rfloor}{n}{q^2}.
$$
\end{theorem}

\begin{proof}
The proof follows by applying Theorems~\ref{theorem_gen_stir1} 
and~\ref{theorem_acyclic_stir1} 
and Lemmas~\ref{lemma_homology} and~\ref{lemma_boundary_stir1}.
\end{proof}

\section{$(q,t)$-Stirling numbers and orthogonality}
\label{section_orthogonality}

In~\cite{Viennot} Viennot has some beautiful results
in which he gave combinatorial bijections for orthogonal polynomials
and their moment generating functions.
One well-known
relation between the ordinary
signed Stirling numbers of the first kind
and Stirling numbers of the second kind
is their orthogonality. 
A bijective proof of the orthogonality of  their $q$-analogues
via $0$-$1$ tableaux was given by de M\'edicis and
Leroux~\cite[Proposition 3.1]{de_Medicis_Leroux_unified}.

There are a number of  two-variable Stirling numbers of the
second kind using bistatistics
on $RG$-words and rook placements.  
See~\cite{Wachs_White} and the references therein. 
Letting $t = 1+q$
we define $(q,t)$-analogues
of the Stirling numbers of the first and second kind.
We show orthogonality holds combinatorially for the
$(q,t)$-version of the Stirling numbers via a
sign-reversing involution on ordered pairs of rook placements
and $RG$-words.

\begin{definition}
Define the 
$(q,t)$-Stirling numbers of the first and second kind by 
\begin{align}
\label{equation_qt_Stirling_first_kind}
      s_{q,t}[n,k] = (-1)^{n-k} \cdot \sum_{T \in \alrook{n}{n-k}}
                      q^{\below(T)} \cdot t^{\nrow(T)}
\end{align}
and
\begin{align}
\label{equation_qt_Stirling_second_kind}
      S_{q,t}[n,k] =  \sum_{w \in \Allow(n,k)}
                      q^{A(w)} \cdot t^{B(w)}.
\end{align}
\end{definition}

For what follows, let
\begin{equation}
     [k]_{q,t}  =
     \begin{cases}
      (q^{k-2} + q^{k-4} + \cdots + 1) \cdot t
     & \text{ when $k$ is even,}\\
       q^{k-1} + (q^{k-3} + q^{k-5} + \cdots + 1) \cdot t
     & \text{ when $k$ is odd.}
     \end{cases}
\label{equation_q_t_analogue}
\end{equation}

\begin{corollary}
The $(q,t)$-analogue of Stirling numbers of the first
and second kind
satisfy the following recurrences:
\begin{equation}
\label{equation_signed_Stirling_first_recurrence}
     s_{q,t}[n,k] = s_{q,t}[n-1,k-1] - [n-1]_{q,t} \cdot s_{q,t}[n-1,k]
     \:\:\:\mbox{     for $n \geq 1$ and $1 \leq k \leq n$,}
\end{equation}
and
\begin{align}
\label{equation_qt_Stirling_second_recurrence}
   \qtStirling{n}{k} = \qtStirling{n-1}{k-1}
           + [k]_{q,t} 
             \cdot \qtStirling{n-1}{k}
     \:\:\:\mbox{     for $n \geq 1$ and $1 \leq k \leq n$}
\end{align}
with
initial conditions 
$s_{q,t}[n,0] = \delta_{n,0}$ and
$\qtStirling{n}{0} = \delta_{n,0}$.
For $k > n$, we set
$s_{q,t}[n,k] = \qtStirling{n}{k} = 0$.
\end{corollary}
\begin{proof}
Immediate from Theorem~\ref{theorem_q_Stirling_second_allowable}
and Theorem~\ref{theorem_q_Stirling_first}.
\end{proof}

Recall the generating polynomials for the $q$-Stirling numbers are
\begin{equation}
\label{equation_q_Stirling_generating}
	(x)_{n,q} = \sum_{k=0}^n s_q[n,k] \cdot x^k
\:\:\: \mbox{    and    }  \:\:\:
	x^n = \sum_{k = 0}^n S_q[n,k] \cdot (x)_{k,q}.
\end{equation}
where
the $q$-analogue of the $k$th falling factorial of $x$
is given by
$$
   (x)_{k,q} = \displaystyle \prod_{m=0}^{k-1} (x-[m]_q).
$$
The expressions
in~(\ref{equation_q_Stirling_generating})
are due to Carlitz~\cite[Section 3]{Carlitz}.  
The case $q=1$ is due to Stirling in 1730
and was
his original definition for the Stirling numbers of the
first and second kind; see~\cite[Pages 8 and 11]{Stirling}.
We can generalize~(\ref{equation_q_Stirling_generating})
to $(q,t)$-polynomials.

\begin{theorem}
\label{theorem_generating_polynomials_q_t}
The generating polynomials for the $(q,t)$-Stirling numbers
are
\begin{equation}
\label{equation_generating_function_one}
	(x)_{n,q,t} = \sum_{k=0}^n s_{q,t}[n,k] \cdot x^k\,,
\end{equation}
and
\begin{equation}
\label{equation_generating_function_two}
	x^n = \sum_{k = 0}^n S_{q,t}[n,k] \cdot (x)_{k,q,t}\,,
\end{equation}
where $(x)_{k,q,t} = \displaystyle \prod_{m=0}^{k-1} (x-[m]_{q,t})$. 
\end{theorem}
\begin{proof}
Both identities follow by induction on $n$. 
It is straightforward to check the case $n=0$,
so suppose
the identities are true for $n-1$. 
Multiply the recurrence~(\ref{equation_signed_Stirling_first_recurrence})
for the signed $(q,t)$-Stirling numbers of the first kind by
by $x^k$ and sum over all $0 \leq k \leq n$ to give
\begin{eqnarray*}
\sum_{k=0}^n s_{q,t}[n,k]\cdot x^k
       &=& \sum_{k=0}^n (s_{q,t}[n-1,k-1] 
                      - [n-1]_{q,t} \cdot s_{q,t}[n-1,k])\cdot x^k \\
	&=& x \cdot \sum_{k=0}^{n-1} s_{q,t}[n-1,k]\cdot x^k - [n-1]_{q,t} \cdot \sum_{k=0}^{n-1} s_{q,t}[n-1,k] \cdot x^k \\
	&=& (x)_{n-1,q,t}\cdot (x-[n-1]_{q,t}) \\
	&=& (x)_{n,q,t}\,,
\end{eqnarray*}
which is the first identity.
For the second identity, multiply
the recurrence~(\ref{equation_qt_Stirling_second_recurrence})
for the $(q,t)$-Stirling number of the second kind
by
$(x)_{k,q,t}$
and sum over all $0 \leq k \leq n$ to give
\begin{eqnarray*}
     \sum_{k=0}^n S_{q,t}[n,k] \cdot  (x)_{k,q,t}
      &=& \sum_{k=0}^n (S_{q,t}[n-1,k-1] + [k]_{q,t} \cdot S_{q,t}[n-1,k]) 
           \cdot (x)_{k,q,t} \\
      &=& \sum_{k=0}^n S_{q,t}[n-1,k-1] \cdot (x)_{k-1,q,t}
           \cdot (x-[k-1]_{q,t}) \\
         && + \sum_{k=0}^n [k]_{q,t} \cdot S_{q,t}[n-1,k] \cdot (x)_{k,q,t} \\
      &=& x \cdot \sum_{k=0}^{n-1} S_{q,t}[n-1,k] \cdot (x)_{k,q,t} \\
      && - \sum_{k=0}^n [k-1]_{q,t} \cdot S_{q,t}[n-1,k-1] 
        + \sum_{k=0}^n [k]_{q,t} \cdot S_{q,t}[n-1,k]\,.
\end{eqnarray*}
The last two summations
cancel each other by shifting indices. 
Apply the induction hypothesis on the remaining summation
yields the desired result.
\end{proof}

\begin{theorem}
\label{theorem_orthogonality}
The $(q,t)$-Stirling numbers are orthogonal, that is, for $m \leq n$
\begin{equation}
\label{equation_orthogonality_first_second}
     \sum_{k=m}^n s_{q,t}[n,k]\cdot S_{q,t}[k,m] = \delta_{m,n}
\end{equation}
and
\begin{equation}
\label{equation_orthogonality_second_first}
     \sum_{k=m}^n S_{q,t}[n,k]\cdot s_{q,t}[k,m] = \delta_{m,n}.
\end{equation}
Furthermore, this orthogonality  holds bijectively.
\end{theorem}

Notice that orthogonality of the $(q,t)$-Stirling numbers
follows immediately 
from Theorem~\ref{theorem_generating_polynomials_q_t}
which gives the change of basis matrices between the
ordered bases
$(1, x, x^2, x^3, \ldots)$ and
$((x)_{0,q,t},$ $(x)_{(1,q,t)},$ $x_{(2,q,t)},$ $x_{(3,q,t)}, \ldots)$
for the polynomial ring
$\Qqq(q,t)[x]$.
We now instead provide a bijective proof.

\begin{proof}
When $m = n$ since $s_{q,t}[n,n] = S_{q,t}[n,n] = 1$, 
both identities are trivial.
Suppose now that $n>m$.
The left-hand side of~(\ref{equation_orthogonality_first_second}) 
is the total weight of the set
$$
	C = \bigcup_{k=m}^{n} \alrook{n}{n-k}\times\allow{k}{m},
$$
where the weight of $(T, w)\in C$ is defined by
$$
	\wt(T, w) = (-1)^{n-k} \cdot \wt(T)\cdot\wt(w).
$$
Here $\wt(w) = q^{A(w)} \cdot t^{B(w)}$ 
and $\wt(T) = q^{\below(T)} \cdot t^{\nrow(T)}$ 
where the statistics $A(\cdot)$,
$B(\cdot)$,
$\below(\cdot)$ and $\nrow(\cdot)$
are defined in Sections~\ref{section_allowable} and~\ref{section_first_kind}.
We wish to show that 
$\wt(C) = \sum_{(T, w)\in C}\wt(T, w) = 0$ by constructing 
a weight-preserving sign-reversing involution 
$\varphi$ on~$C$ with no fixed points.

\begin{figure}
\centering
\subfloat[Example when $l_1 \leq l_2$]
{
\begin{tikzpicture}[>=triangle 45,x=1.0cm,y=1.0cm, scale=0.5]
\clip(-4,-1) rectangle (14.5,5.5);
\fill[color=gray,fill=gray,fill opacity=0.2] (1,5) -- (0,5) -- (0,4) -- (1,4) -- cycle;
\fill[color=gray,fill=gray,fill opacity=0.2] (-1,5) -- (-2,5) -- (-2,4) -- (-1,4) -- cycle;
\fill[color=gray,fill=gray,fill opacity=0.2] (-2,4) -- (-3,4) -- (-3,3) -- (-2,3) -- cycle;
\fill[color=gray,fill=gray,fill opacity=0.2] (-1,4) -- (0,4) -- (0,3) -- (-1,3) -- cycle;
\fill[color=gray,fill=gray,fill opacity=0.2] (-1,3) -- (-2,3) -- (-2,2) -- (-1,2) -- cycle;
\fill[color=gray,fill=gray,fill opacity=0.2] (10,5) -- (11,5) -- (11,4) -- (10,4) -- cycle;
\fill[color=gray,fill=gray,fill opacity=0.2] (10,4) -- (10,3) -- (9,3) -- (9,4) -- cycle;
\fill[color=gray,fill=gray,fill opacity=0.2] (9,5) -- (8,5) -- (8,4) -- (9,4) -- cycle;
\fill[color=gray,fill=gray,fill opacity=0.2] (8,4) -- (7,4) -- (7,3) -- (8,3) -- cycle;
\fill[color=gray,fill=gray,fill opacity=0.2] (8,3) -- (9,3) -- (9,2) -- (8,2) -- cycle;
\fill[color=gray,fill=gray,fill opacity=0.2] (8,2) -- (7,2) -- (7,1) -- (8,1) -- cycle;
\fill[color=gray,fill=gray,fill opacity=0.2] (-2,2) -- (-3,2) -- (-3,1) -- (-2,1) -- cycle;
\draw (-3,5)-- (1,5);
\draw (-3,4)-- (1,4);
\draw (-3,3)-- (0,3);
\draw (-3,2)-- (-1,2);
\draw (-3,1)-- (-2,1);
\draw (-3,5)-- (-3,1);
\draw (-2,5)-- (-2,1);
\draw (-1,5)-- (-1,2);
\draw (0,5)-- (0,3);
\draw (1,5)-- (1,4);
\draw (7,5)-- (11,5);
\draw (7,4)-- (11,4);
\draw (7,3)-- (10,3);
\draw (7,2)-- (9,2);
\draw (7,1)-- (8,1);
\draw (7,5)-- (7,1);
\draw (8,5)-- (8,1);
\draw (9,5)-- (9,2);
\draw (10,5)-- (10,3);
\draw (11,5)-- (11,4);
\draw [line width=1.25pt] (-1,4)-- (0,3);
\draw [line width=1.25pt] (0,4)-- (-1,3);
\draw [line width=1.25pt] (-2,5)-- (-1,4);
\draw [line width=1.25pt] (-1,5)-- (-2,4);
\draw [line width=1.25pt] (8,5)-- (9,4);
\draw [line width=1.25pt] (9,5)-- (8,4);
\draw (1.08,4.02) node[anchor=north west] {$\times \:\:\:\: 121$};
\draw (11.08,4.02) node[anchor=north west] {$\times \:\:\:\: 1211$};
\draw (4.0,3.76) node[anchor=north west, font=\huge] {$\longmapsto$};
\draw (-1.5,0.7) node[anchor=north west] {$l_1 = 2, l_2 = 2$};
\draw (7.5,0.7) node[anchor=north west] {$l_1 = 3, l_2 = 2$};
\end{tikzpicture}
}

\qquad
\subfloat[Example when $l_1 > l_2$.]
{
\begin{tikzpicture}[>=triangle 45,x=1.0cm,y=1.0cm, scale = 0.5]
\clip(-3.4,-1) rectangle (15.5,7);
\fill[color=gray,fill=gray,fill opacity=0.2] (1,5) -- (0,5) -- (0,4) -- (1,4) -- cycle;
\fill[color=gray,fill=gray,fill opacity=0.2] (-1,5) -- (-2,5) -- (-2,4) -- (-1,4) -- cycle;
\fill[color=gray,fill=gray,fill opacity=0.2] (-2,4) -- (-3,4) -- (-3,3) -- (-2,3) -- cycle;
\fill[color=gray,fill=gray,fill opacity=0.2] (-1,4) -- (0,4) -- (0,3) -- (-1,3) -- cycle;
\fill[color=gray,fill=gray,fill opacity=0.2] (-1,3) -- (-2,3) -- (-2,2) -- (-1,2) -- cycle;
\fill[color=gray,fill=gray,fill opacity=0.2] (10,5) -- (11,5) -- (11,4) -- (10,4) -- cycle;
\fill[color=gray,fill=gray,fill opacity=0.2] (10,4) -- (10,3) -- (9,3) -- (9,4) -- cycle;
\fill[color=gray,fill=gray,fill opacity=0.2] (9,5) -- (8,5) -- (8,4) -- (9,4) -- cycle;
\fill[color=gray,fill=gray,fill opacity=0.2] (8,4) -- (7,4) -- (7,3) -- (8,3) -- cycle;
\fill[color=gray,fill=gray,fill opacity=0.2] (8,3) -- (9,3) -- (9,2) -- (8,2) -- cycle;
\fill[color=gray,fill=gray,fill opacity=0.2] (8,2) -- (7,2) -- (7,1) -- (8,1) -- cycle;
\fill[color=gray,fill=gray,fill opacity=0.2] (-2,2) -- (-3,2) -- (-3,1) -- (-2,1) -- cycle;
\draw (-3,5)-- (1,5);
\draw (-3,4)-- (1,4);
\draw (-3,3)-- (0,3);
\draw (-3,2)-- (-1,2);
\draw (-3,1)-- (-2,1);
\draw (-3,5)-- (-3,1);
\draw (-2,5)-- (-2,1);
\draw (-1,5)-- (-1,2);
\draw (0,5)-- (0,3);
\draw (1,5)-- (1,4);
\draw (7,5)-- (11,5);
\draw (7,4)-- (11,4);
\draw (7,3)-- (10,3);
\draw (7,2)-- (9,2);
\draw (7,1)-- (8,1);
\draw (7,5)-- (7,1);
\draw (8,5)-- (8,1);
\draw (9,5)-- (9,2);
\draw (10,5)-- (10,3);
\draw (11,5)-- (11,4);
\draw [line width=1.25pt] (-2,3)-- (-1,2);
\draw [line width=1.25pt] (-1,3)-- (-2,2);
\draw [line width=1.25pt] (8,3)-- (9,2);
\draw [line width=1.25pt] (9,3)-- (8,2);
\draw [line width=1.25pt] (9,4)-- (10,3);
\draw [line width=1.25pt] (10,4)-- (9,3);
\draw (1.08,4.02) node[anchor=north west] {$\times \:\:\:\: 1211$};
\draw (11.08,4.02) node[anchor=north west] {$\times \:\:\:\: 121$};
\draw (4.0,3.76) node[anchor=north west, font=\huge] {$\longmapsto$};
\draw (-1.5,0.7) node[anchor=north west] {$l_1 = 3, l_2 = 2$};
\draw (7.5,0.7) node[anchor=north west] {$l_1 = 2, l_2 = 2$};
\end{tikzpicture}
}
\caption{Examples of the bijection proving
the identity~(\ref{equation_orthogonality_first_second}).}
\label{figure_first_orthogonality_proof}
\end{figure}
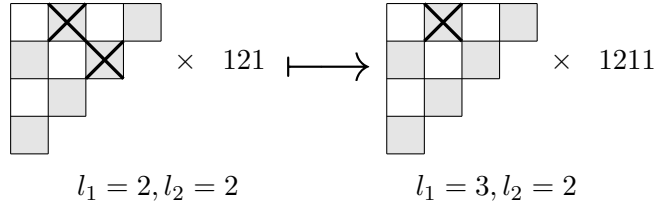
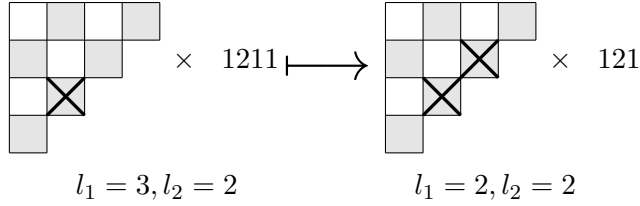

For any pair $(T, w)\in \alrook{n}{n-k}\times\allow{k}{m}$, 
define the map $\varphi$ as follows. 
Label the columns of
$T \in \alrook{n}{n-k}$
from right to
left with $1$ through $n-1$.
Let $l_1$ be the 
label of the
rightmost column in $T$ that has a rook. 
If $T$ has no rooks, let $l_1 = \infty$. 
Denote by
$\rb(T)$ the number of squares below the rightmost rook in $T$.
If
$l_1 = \infty$, 
let $\rb(T) = 0$. 
For $w \in \allow{k}{m}$, let $r$
be the first repeating (odd) integer reading 
the entries of $w$ from left to right, 
and let $l_2$ denote the number appearing
to the left of the entry $r$ in the $RG$-word $w$. 
If there is no repeating integer, let $l_2 = \infty$. 
Note that $\rb(T)$ must be even.

If $l_1 \leq l_2$, remove the rightmost rook in $T$ to form the 
rook placement $T'$. Insert the entry
$\rb(T)+1$  to the right of the entry
$l_1$ 
to obtain the word $w'$.   Since $l_1 \leq l_2$,
$\rb(T)+1 \leq l_1 \leq l_2$ and $\rb(T)+1$ is odd, so we have $w'$ is
an allowable word of length $k+1$. Hence
$(t', w') \in \alrook{n}{n-k-1}\times \allow{k+1}{m}$. Also
since we removed the rightmost rook in $T$ to obtain $T'$, we know
$\wt(T) = q^{l_1}\cdot \wt(T')$ if $\rb(T)+1 = l_1$, that is, the
rightmost rook is in the first row, or that
$\wt(T) = q^{\rb(T)}\cdot t \cdot \wt(T')$ if $\rb(T) + 1 < l_1$, 
that is, the rightmost rook is
not in the first row.  
We also know that 
$\wt(w') = q^{l_1-1}\cdot \wt(w)$ if $l_1 = \rb(T)+1$, or 
$\wt(w') =q^{\rb(T)}\cdot t \cdot \wt(w)$ if $\rb(T)+1<l_1$. Thus
$\wt(T', w') = (-1)^{n-k-1} \cdot \wt(T')\cdot\wt(w') =
-\wt(T, w)$.

On the other hand, if $l_1 > l_2$, 
delete the entry  $r$ in $w$ 
to obtain $w'$.
In column $l_2$ of $T$ add 
a rook so that there are  $r-1$ empty squares below it.
Similarly, one can check that 
$(T', w') \in \alrook{n}{n-k+1} \times \allow{k-1}{m}$ and $\wt(T', w') = -\wt(T, w)$.

Since all pairs $(T, w) \in \alrook{n}{n-k} \times \allow{k}{m}$ 
are mapped under $\varphi$, there are no fixed
points in $C$, hence~(\ref{equation_orthogonality_first_second}) is true.

The proof of the 
second identity~(\ref{equation_orthogonality_second_first}) 
follows in a similar fashion.
The left-hand side of~(\ref{equation_orthogonality_second_first}) 
is the total weight of the set
$$
	D = \bigcup_{k = m}^{n} \allow{n}{k} \times \alrook{k}{k-m}\,,
$$
where $\wt(w, T) = (-1)^{k-m} \cdot \wt(w) \cdot \wt(T)$.
We show that 
$\wt(D) = \sum_{(w, T)\in D}\wt(w, T) = 0$ by constructing 
a weight-preserving sign-reversing involution $\psi$ on~$D$ with no fixed points.

For $(w, T)\in \allow{n}{k} \times \alrook{k}{k-m}$, 
define the following.
Let $w_i = r_1$ be the last repeated odd integer
in $w$ reading from left to right,
and let $l_1$ be the maximum entry in $w$ 
occurring before $w_i$. 
If there is no repeated entry in $w$, let $l_1 = 0$. 
Let $l_2$ be the label of the leftmost column in $T$ with a rook in it 
and let $r_2$ be the number of squares above that rook. 
If there are no rooks in $T$ let $l_2 = 0$.
As before, we are labeling the columns right to left with 
$1$ through $n-1$.

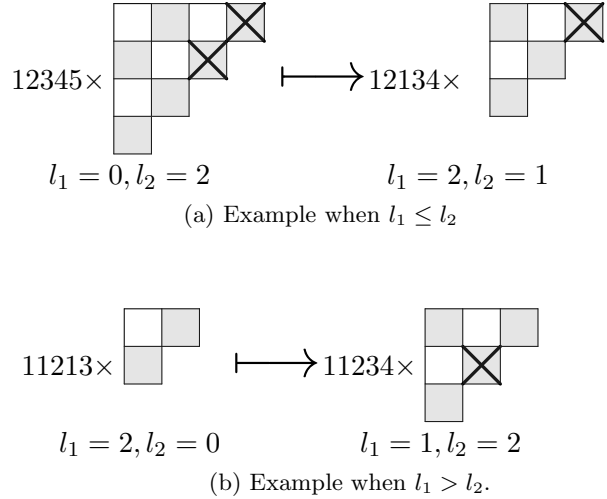
\begin{figure}[tb]
\centering
\subfloat[Example when $l_1 \leq l_2$]
{
\begin{tikzpicture}[line cap=round,line join=round,>=triangle 45,x=1.0cm,y=1.0cm,scale=0.5]
\clip(-5,1) rectangle (12,7);
\draw (-2.0,6.0)-- (2.0,6.0);
\draw (-2.0,5.0)-- (2.0,5.0);
\draw (-2.0,4.0)-- (1.0,4.0);
\draw (-2.0,3.0)-- (0.0,3.0);
\draw (-2.0,2.0)-- (-1.0,2.0);
\draw (-2.0,6.0)-- (-2.0,2.0);
\draw (-1.0,6.0)-- (-1.0,2.0);
\draw (0.0,6.0)-- (0.0,3.0);
\draw (1.0,6.0)-- (1.0,4.0);
\draw (2.0,6.0)-- (2.0,5.0);

\draw [line width=1.25pt] (1.0,6.0)-- (2.0,5.0);
\draw [line width=1.25pt]  (2.0,6.0)-- (1.0,5.0);
\draw [line width=1.25pt]  (0.0,5.0)-- (1.0,4.0);
\draw [line width=1.25pt]  (1.0,5.0)-- (0.0,4.0);

\draw (8.0,6.0)-- (11.0,6.0);
\draw (8.0,5.0)-- (11.0,5.0);
\draw (8.0,4.0)-- (10.0,4.0);
\draw (8.0,3.0)-- (9.0,3.0);
\draw (8.0,6.0)-- (8.0,3.0);
\draw (9.0,6.0)-- (9.0,3.0);
\draw (10.0,6.0)-- (10.0,4.0);
\draw (11.0,6.0)-- (11.0,5.0);
\draw [line width=1.25pt] (10.0,6.0)-- (11.0,5.0);
\draw [line width=1.25 pt] (11.0,6.0)-- (10.0,5.0);
\fill[color=gray,fill=gray,fill opacity=0.2] (1.0,6.0) -- (1.0,5.0) -- (2.0,5.0) -- (2.0,6.0) -- cycle;
\fill[color=gray,fill=gray,fill opacity=0.2] (-1.0,6.0) -- (-1.0,5.0) -- (0.0,5.0) -- (0.0,6.0) -- cycle;
\fill[color=gray,fill=gray,fill opacity=0.2] (0.0,5.0) -- (-0.0,4.0) -- (1.0,4.0) -- (1.0,5.0) -- cycle;
\fill[color=gray,fill=gray,fill opacity=0.2] (-1.0,4.0) -- (-1.0,3.0) -- (0.0,3.0) -- (0.0,4.0) -- cycle;
\fill[color=gray,fill=gray,fill opacity=0.2] (-2.0,3.0) -- (-2.0,2.0) -- (-1.0,2.0) -- (-1.0,3.0) -- cycle;
\fill[color=gray,fill=gray,fill opacity=0.2] (-2.0,5.0) -- (-2.0,4.0) -- (-1.0,4.0) -- (-1.0,5.0) -- cycle;
\fill[color=gray,fill=gray,fill opacity=0.2] (10.0,6.0) -- (11.0,6.0) -- (11.0,5.0) -- (10.0,5.0) -- cycle;
\fill[color=gray,fill=gray,fill opacity=0.2] (9.0,6.0) -- (8.0,6.0) -- (8.0,5.0) -- (9.0,5.0) -- cycle;
\fill[color=gray,fill=gray,fill opacity=0.2] (9.0,5.0) -- (10.0,5.0) -- (10.0,4.0) -- (9.0,4.0) -- cycle;
\fill[color=gray,fill=gray,fill opacity=0.2] (8.0,4.0) -- (9.0,4.0) -- (9.0,3.0) -- (8.0,3.0) -- cycle;
\draw (-5,4.5) node[anchor=north west] {$12345 \times$};
\draw (2.1,4.5) node[anchor=north west, font=\huge] {$\longmapsto$};
\draw (4.5,4.5) node[anchor=north west] {$12134 \times$};
\draw (-4.0,2.0) node[anchor=north west] {$l_1 = 0, l_2 = 2$};
\draw (5.0,2.0) node[anchor=north west] {$l_1 = 2, l_2 = 1$};
\end{tikzpicture}
}

\qquad
\subfloat[Example when $l_1 > l_2$.]
{
\begin{tikzpicture}[line cap=round,line join=round,>=triangle 45,x=1.0cm,y=1.0cm,scale=0.5]
\clip(-8.0,0.0) rectangle (10,6);
\draw (-5.0,4.0)-- (-3.0,4.0);
\draw (-5.0,3.0)-- (-3.0,3.0);
\draw (-5.0,2.0)-- (-4.0,2.0);
\draw (-5.0,4.0)-- (-5.0,2.0);
\draw (-4.0,4.0)-- (-4.0,2.0);
\draw (-3.0,4.0)-- (-3.0,3.0);
\draw (3.0,4.0)-- (6.0,4.0);
\draw (3.0,3.0)-- (6.0,3.0);
\draw (3.0,2.0)-- (5.0,2.0);
\draw (3.0,1.0)-- (4.0,1.0);
\draw (3.0,4.0)-- (3.0,1.0);
\draw (4.0,4.0)-- (4.0,1.0);
\draw (5.0,4.0)-- (5.0,2.0);
\draw (6.0,4.0)-- (6.0,3.0);
\draw [line width=1.25pt] (5.0,2.0)-- (4.0,3.0);
\draw [line width=1.25pt] (5.0,3.0)-- (4.0,2.0);
\fill[color=gray,fill=gray,fill opacity=0.2] (-4.0,4.0) -- (-3.0,4.0) -- (-3.0,3.0) -- (-4.0,3.0) -- cycle;
\fill[color=gray,fill=gray,fill opacity=0.2] (-5.0,3.0) -- (-4.0,3.0) -- (-4.0,2.0) -- (-5.0,2.0) -- cycle;
\fill[color=gray,fill=gray,fill opacity=0.2] (5.0,4.0) -- (6.0,4.0) -- (6.0,3.0) -- (5.0,3.0) -- cycle;
\fill[color=gray,fill=gray,fill opacity=0.2] (4.0,3.0) -- (5.0,3.0) -- (5.0,2.0) -- (4.0,2.0) -- cycle;
\fill[color=gray,fill=gray,fill opacity=0.2] (4.0,4.0) -- (3.0,4.0) -- (3.0,3.0) -- (4.0,3.0) -- cycle;
\fill[color=gray,fill=gray,fill opacity=0.2] (3.0,2.0) -- (4.0,2.0) -- (4.0,1.0) -- (3.0,1.0) -- cycle;
\draw (-8.0,3.0) node[anchor=north west] {$11213 \times$};
\draw (-2.4,3.0) node[anchor=north west, font=\huge] {$\longmapsto$};
\draw (0.0,3.0) node[anchor=north west] {$11234 \times$};
\draw (-7.0,1.0) node[anchor=north west] {$l_1 = 2, l_2 = 0$};
\draw (1,1.0) node[anchor=north west] {$l_1 = 1, l_2 = 2$};
\end{tikzpicture}
}
\caption{Examples of the bijection proving the identity~(\ref{equation_orthogonality_second_first}).}
\label{figure_second_orthogonality_proof}
\end{figure}

The bijection is built as follows.
If $l_1 > l_2$, raise $w_i = r_1$ to $l_1 + 1$ and 
increase all of the entries to the right of $w_i$ by $1$.
Denote the new word by $w'$. 
Since $w_i$ is the last repeated odd entry, 
the $RG$-word $w$ is of the form 
$w = \cdots l_1 \cdots r_1 (l_1+1) (l_1+2) \cdots k$. 
Then by definition, 
the new word $w'$ is of the form 
$w' = \cdots l_1 \cdots (l_1+1)(l_1+2)(l_1+3) \cdots (k+1)$. 
This still is an allowable word 
since the first $i-1$ entries in $w'$ are the same 
as that in $w$ and 
the remaining entries form an  increasing sequence. 
So $w' \in \allow{n}{k+1}$. 
Also, in $w$ the entries after $w_i$ 
do not contribute to $\wt(w)$ 
since there are no repeated entries. 
When $w_i$ is raised to $l_1+1$, 
the weight loss  is $q^{r_1 - 1}$ if $r_1 = l_1$ 
or $q^{r_1 - 1} \cdot t$ if $r_1 < l_1$. 
In the staircase board $T$, 
form a new rook placement $T'$ by
first adding a column of length $k$ to the left, 
and then placing a rook in column $l_1$ counting from right to left 
such that there are $r_1-1$ squares below the rook. 
Clearly $T'$ has $k$ columns and $k+1-m$ rooks. 
Since the new rook was placed 
so that there are now an
even number of squares below it, 
this rook is in a shaded square. 
Also since $l_1 > l_2$, 
there is no other rook in column~$l_1$. 
Hence $T' \in \alrook{k+1}{k+1-m}$.
Observe when we add a rook to obtain~$T'$, 
if the new rook is added in the first row, that is, $r_1 = l_1$ then
the weight is increased by~$q^{r_1-1}$.
If the new rook is not in the first row, that is,
$r_1 < l_1$ then 
the weight is increased by $q^{r_1-1} \cdot t$.
Hence $\wt(w', T') = -\wt(w, T)$.

If $l_1 \leq l_2$, replace the entry $w_j = l_2+1$ in 
$w$ by $l_2 - r_2$ and subtract $1$ from all 
of the entries to the right of $w_j$ to obtain $w'$. 
Since $w = \cdots l_1 \cdots r_1 (l_1+1)\cdots k$ 
and $l_1 \leq l_2 \leq k-1$, 
we have that $w_j = l_2+1$ appears to the right of $w_i$ 
and hence such an entry is unique. 
Also $r_2 + 1 \leq l_2$ gives $l_2 - r_2 \geq 1$.
This difference is always odd by the fact 
that the rook is in a shaded square. 
So $w' = \cdots l_1 \cdots l_2 (l_2-r_2) (l_2+1) \cdots (k-1)$ 
is an $RG$-word with even integers appearing just once, 
hence $w' \in \allow{n}{k-1}$. 
The entry $w_{j-1}' = l_2$, 
and $w_j' = l_2 - r_2$ contributes 
a weight of
$q^{l_2-r_2-1}$ if $l_2 = l_2 - r_2$, that is, $r_2 = 0$ 
or $q^{l_2-r_2-1} \cdot t$ if $r_2 > 0$.
Delete the column $l_2$ in $T$ 
and delete one square from the bottom 
in all columns to the left of 
column $l_2$ to make the new staircase chessboard $T'$.
It is straightforward to check that
$T' \in \alrook{(k-1}{k-1-m}$. 
Deleting the rook in $T$ will decrease its 
weight by $q^{l_2-(r_2+1)}$ if the rook is in the first row, 
that is, $r_2 = 0$
or by $q^{l_2-r_2-1} \cdot t$ if the rook is not in the first row,
that is, $r_2 > 0$. 
Hence $\wt(w', T') = -\wt(w, T)$.
The map we described is a 
weight-preserving sign-reversing involution with no fixed points,
so the orthogonality in~(\ref{equation_orthogonality_second_first})
follows.
\end{proof}

See Figures~\ref{figure_first_orthogonality_proof}
and~\ref{figure_second_orthogonality_proof}
for examples 
of the bijections
occurring in the proof of Theorem~\ref{theorem_orthogonality}.

\section{Concluding remarks}

The Stirling numbers of the first kind and second kind
are specializations of 
the homogeneous and elementary symmetric functions:
\begin{align}
S(n,k) = h_{n-k}(x_1, \ldots, x_k),
\:\:\:\:\:\:\:\:\:\:
c(n,k) =  e_{n-k}(x_1, \ldots, x_{n-1}),
\end{align}
where 
$x_m = m$.
The $q$-Stirling numbers are also specializations of
these Schur functions with
$x_m = [m]_q$.
See~\cite[Chapter I, Section 2, Example 11]{Macdonald}.
For the 
$(q,t)$-versions take
$x_m = [m]_{q,t}$ as defined in~(\ref{equation_q_t_analogue}).
A more general statement of orthogonality is
\begin{align}
   \sum_{k=j}^n (-1)^{n-k} \cdot e_{n-k}(x_1, \ldots, x_{n-1})  \cdot
                  h_{k-j}(x_1, \ldots, x_{j})  = \delta_{n,j}.
\end{align}
The specializations imply orthogonality of
the $(q,t)$-Stirling numbers, though not combinatorially
as in Theorem~\ref{theorem_orthogonality}.
It remains to find a combinatorial proof of
Theorem~\ref{theorem_generating_polynomials_q_t}.

Stembridge's $q=-1$ phenomenon~\cite{Stembridge, Stembridge_canonical},
and the more general
cyclic sieving phenomenon of 
Reiner, Stanton and White~\cite{Reiner_Stanton_White}
count symmetry classes in 
combinatorial objects
by evaluating their $q$-generating series at
a primitive root of unity. Is there a cyclic sieving phenomenon for the $q$-Stirling numbers of the first and second kind?

Are there other classical $q$-analogues
which can be viewed naturally as $q$-$(1+q)$-analogues
as in Goals~\ref{goal_one} and \ref{goal_two}?
Ehrenborg and Readdy~\cite{Ehrenborg_Readdy_Gaussian} 
have recently discovered
a {\em symmetric}
$q$-$(1+q)$-analogue of
the $q$-binomial which
is more compact than the 
Fu et al construction.

Garsia and Remmel~\cite{Garsia_Remmel} have a more
general notion of the $q$-Stirling number of the second kind
as enumerating non-attacking rooks on a general Ferrers' board.
This will be the subject of another paper.

It would be interesting to look deeper into the poset
structure of the
Stirling posets of the first and second kind,
such as the interval structure and the $f$- and $h$-vectors
of each poset. Park has a notion of the Stirling poset which
arises from the theory of $P$-partitions~\cite{Park}.
It has no connection with the Stirling posets
in this paper.

The $q$-binomial
has the combinatorial interpretation of 
counting certain
subspaces over a finite
field with $q$ elements 
as well as the corresponding subspace lattice.
Milne~\cite{Milne_restricted} has
an interpretation of the $q$-Stirling number of the second
kind as sequences of lines in a vector space over the
finite field with $q$ elements. Is there an analogous interpretation for the $(q, t)$-Stirling numbers of the second kind? 
Bennett, Dempsey and Sagan~\cite{Bennett_Dempsey_Sagan} construct 
families of posets which include Milne's
construction.
One would like a similar
construction for the $q$-Stirling 
numbers of the first kind.

In~\cite{Wachs_White} 
Wachs and White have 
discovered
many other statistics on $RG$-words which generate
the $q$-Stirling numbers. 
In particular, 
their $ls$ and $lb$ statistics are defined by 
$ls(w) = \prod_{i=1}^n q^{w_i - 1}$ and 
$lb(w) = \prod_{i = 1}^n lb_i(w)$ where
$lb_i(w) = q^{m_{i-1}-w_i}$ if $m_{i-1} \geq w_{i}$
and
$lb_i(w) = 1$ if $m_{i-1} < w_i$.
The 
$ls$ statistic and the $\wt$ statistic
in~(\ref{equation_Stirling_second_weight})
are related by
$ls(w) = q^{\binom{k}{2}} \cdot \wt(w)$.
The authors are currently looking 
at these statistics,
as well as White's
interpolations~\cite{White} between
these statistics,
in view of the first Goal~\ref{goal_one},
as well as poset theoretic and homological consequences
of Goal~\ref{goal_two}.
The first author has considered the $q$-binomial
via the major
index in terms of this research program~\cite{Cai_major_index}.

\section*{Acknowledgements}

The first author was partially supported by National
Security Agency grant H98230-13-1-0280.
This work was partially supported by a grant
from the Simons Foundation
(\#206001 to Margaret Readdy).
The second author would like to thank the Princeton University
Mathematics Department
where this work was completed.
The authors thank Dennis Stanton
for conversations when this project was being
initiated,
and
Richard Ehrenborg,
who provided many helpful comments
on the exposition.
Thanks also to 
Jim Haglund,
Vic Reiner,
Dennis Stanton,
Michelle Wachs
and
Dennis White
for their comments 
regarding future research directions,
and to Doron Zeilberger for giving historical
references for $q$-analogues.

\newcommand{\journal}[6]{{\sc #1,} #2, {\it #3} {\bf #4} (#5), #6.}
\newcommand{\book}[4]{{\sc #1,} ``#2,'' #3, #4.}
\newcommand{\bookf}[5]{{\sc #1,} ``#2,'' #3, #4, #5.}
\newcommand{\books}[6]{{\sc #1,} ``#2,'' #3, #4, #5, #6.}
\newcommand{\collection}[6]{{\sc #1,}  #2, #3, in {\it #4}, #5, #6.}
\newcommand{\thesis}[4]{{\sc #1,} ``#2,'' Doctoral dissertation, #3, #4.}
\newcommand{\springer}[4]{{\sc #1,} ``#2,'' Lecture Notes in Math.,
                          Vol.\ #3, Springer-Verlag, Berlin, #4.}
\newcommand{\preprint}[3]{{\sc #1,} #2, preprint #3.}
\newcommand{\preparation}[2]{{\sc #1,} #2, in preparation.}
\newcommand{\appear}[3]{{\sc #1,} #2, to appear in {\it #3}}
\newcommand{\submitted}[3]{{\sc #1,} #2, submitted to {\it #3}}
\newcommand{\JCTA}{J.\ Combin.\ Theory Ser.\ A}
\newcommand{\AdvancesinMathematics}{Adv.\ Math.}
\newcommand{\JournalofAlgebraicCombinatorics}{J.\ Algebraic Combin.}

\newcommand{\communication}[1]{{\sc #1,} personal communication.}


\bigskip

\noindent
{\em Y.\ Cai and M.\ Readdy,
Department of Mathematics,
University of Kentucky,
Lexington, KY 40506,}
{\tt yue.cai@uky.edu}, {\tt margaret.readdy@uky.edu}

\end{document}